\documentclass[12pt]{amsart}
\usepackage{amsfonts}
\usepackage{amsmath}
\usepackage{amssymb}
\usepackage[margin=1in]{geometry}
\usepackage{tikz}
\usepackage{amsthm}   							
\usepackage{mathtools}							
\usepackage{commath}							
\usepackage{enumitem}  							

\usepackage[english]{babel}						

\makeatletter
\@namedef{subjclassname@2020}{\textup{2020} Mathematics Subject Classification}
\makeatother

\newtheorem{theorem}{Theorem}[section]

\newtheorem{lemma}[theorem]{Lemma}

\theoremstyle{definition}

\newtheorem{problem}[theorem]{Problem}

\numberwithin{equation}{section}

\newcommand{\expm}{\exp^{\ast}}
\newcommand{\nconv}{^{\ast n}}

\newcommand{\Z}{\mathbb{Z}}

\newcommand{\R}{\mathbb{R}}
\newcommand{\C}{\mathbb{C}}

\renewcommand{\Re}{\operatorname{Re}}
\renewcommand{\Im}{\operatorname{Im}}

\newcommand{\I}{\mathrm{i}}
\newcommand{\e}{\mathrm{e}}

\newcommand{\eps}{\varepsilon}
\newcommand{\vphi}{\varphi}

\newcommand{\MP}{\mathcal{P}}
\newcommand{\MN}{\mathcal{N}}

\newcommand{\LHS}{\text{LHS}}
\newcommand{\RHS}{\text{RHS}}

\DeclareMathOperator{\Li}{Li}
\DeclareMathOperator{\li}{li}
\DeclareMathOperator*{\res}{Res}

\DeclareMathOperator{\length}{length}
\DeclareMathOperator{\supp}{supp}

\begin{document}

\title[Optimal Malliavin-type remainder]{The optimal Malliavin-type remainder for Beurling generalized integers}

\author[F. Broucke]{Frederik Broucke}
\thanks{F. Broucke was supported by the Ghent University BOF-grant 01J04017}

\author[G.~Debruyne]{Gregory Debruyne}
\thanks{G.~Debruyne acknowledges support by Postdoctoral Research Fellowships of the Research Foundation--Flanders (grant number
12X9719N) and the Belgian American Educational Foundation. The latter one allowed him to do part of this research at the University of Illinois at Urbana-Champaign.} 

\author[J. Vindas]{Jasson Vindas}
\thanks {J. Vindas was partly supported by Ghent University through the BOF-grant 01J04017 and by the Research Foundation--Flanders through the FWO-grant 1510119N}

\address{Department of Mathematics: Analysis, Logic and Discrete Mathematics\\ Ghent University\\ Krijgslaan 281\\ 9000 Gent\\ Belgium}
\email{fabrouck.broucke@ugent.be}
\email{gregory.debruyne@ugent.be}
\email{jasson.vindas@ugent.be}


\begin{abstract} We establish the optimal order of Malliavin-type remainders in the asymptotic density approximation formula for Beurling generalized integers. Given $\alpha\in (0,1]$ and $c>0$ (with $c\leq 1$ if $\alpha=1$), a generalized number system is constructed with Riemann prime counting function 
$
\Pi(x)= \Li(x)+ O(x\exp (-c \log^{\alpha} x ) +\log_{2}x),
$
and whose integer counting function satisfies the extremal oscillation estimate $N(x)=\rho x + \Omega_{\pm}(x\exp(- c'(\log x\log_{2} x)^{\frac{\alpha}{\alpha+1}})$
 for any $c'>(c(\alpha+1))^{\frac{1}{\alpha+1}}$, where $\rho>0$ is its asymptotic density.  In particular, this  improves and extends upon the earlier work [Adv. Math. 370  (2020), Article 107240].   \end{abstract}

\maketitle

\section{Introduction}

In this paper we study the optimality of Malliavin-type remainders in the asymptotic density approximation formula for Beurling generalized integers, a problem that has its roots in a long-standing open question of Bateman and Diamond \cite[13B, p.~199]{BD69}. Let $\mathcal{P}:\:p_{1}\leq p_{2}\leq \dots$ be a Beurling generalized prime system, namely, an unbounded and non-decreasing sequence of positive real numbers satisfying $p_{1}>1$, and let $\mathcal{N}$ be its associated system of generalized integers, that is, the multiplicative semigroup generated by 1 and $\mathcal{P}$   \cite{BD69,Beurling1937,DiamondZhangbook}. We consider the functions $\pi(x)$ and $N(x)$ counting the number of generalized primes and integers, respectively, not exceeding $x$.

Malliavin discovered \cite{Malliavin} that the two asymptotic relations 
\begin{equation}
\label{eq: Malliavin PNT}
\tag{P$_{\alpha}$}
\pi(x)= \Li(x)+ O(x\exp (-c \log^{\alpha} x ))
\end{equation}
and
\begin{equation}
\label{eq: Malliavin density}
\tag{N$_{\beta}$}
N(x)= \rho x+ O(x\exp (-c' \log^{\beta} x )) \qquad (\rho>0),
\end{equation}
for some $c>0$ and $c'>0$, are closely related to each other in the sense that if  \eqref{eq: Malliavin density} holds for a given $0<\beta\leq1$, then (P$_{\alpha^{\ast}}$) is satisfied for some $\alpha^\ast$, and vice versa the relation \eqref{eq: Malliavin PNT} for a given $0<\alpha\leq1$ ensures that (N$_{\beta^{\ast}}$) holds for a certain $\beta^{\ast}$. A natural question is then what the optimal error terms of Malliavin-type are. Writing $\alpha^{\ast}(\beta)$ and $\beta^{\ast}(\alpha)$ for the best possible\footnote{That is, the suprema over all admissible values $\alpha^{\ast}$ and $\beta^{\ast}$ in these implications, respectively.} exponents in these implications, we have:
 
 \begin{problem}\label{Malliavin open problem} Given any $\alpha,\beta\in (0,1]$, find the best exponents $\alpha^{\ast}(\beta)$ and $\beta^{\ast}(\alpha)$. \end{problem}

So far, there are only 
two instances where a solution to Problem \ref{Malliavin open problem} is known. In 2006, Diamond, Montgomery, and Vorhauer \cite{DiamondMontgomeryVorhauer} (cf. \cite{Zhang2007}) demonstrated that $\alpha^{\ast}(1) = 1/2$, while in our recent work [7] we have shown that $\beta^{\ast}(1)=1/2$. The former result proves that the de la Vall\'{e}e Poussin remainder is best possible in Landau's classical PNT \cite{Landau1903}, whereas the latter one yields the optimality of a theorem of Hilberdink and Lapidus \cite{HilberdinkLapidus2006}.

We shall solve here Problem \ref{Malliavin open problem} for any value $\alpha\in (0,1]$. Improving upon Malliavin's results, Diamond \cite{Diamond1970} (cf. \cite{HilberdinkLapidus2006}) established the lower bound $\beta^{\ast}(\alpha)\geq \alpha/(1+\alpha)$. We will prove the reverse inequality:

\begin{theorem}
\label{theorem Malliavin exponent} We have $\beta^{\ast}(\alpha)=\alpha/(1+\alpha)$ for any $\alpha\in (0,1]$.
\end{theorem}

Our main result actually supplies more accurate information and, in particular, it exhibits the best possible value of the constant $c'$ in \eqref{eq: Malliavin density}. In order to explain it, let us first state Diamond's result in a refined form, showing the explicit dependency of the constant $c'$ on $c$ and $\alpha$. We write $\log_{k} x$ for the $k$ times iterated logarithm. The Riemann prime counting function of the generalized number system naturally occurs in our considerations\footnote{If $0<\alpha<1$ or if $\alpha=1$ and $c\leq 1/2$, the functions $\Pi$ and $\pi$ are interchangeable in \eqref{eq: Diamond sharp PNT} since $\Pi(x)=\pi(x)+O(x^{1/2})$; otherwise one must work with $\Pi$.}; as in classical number theory, it is defined as $\Pi(x)=\sum_{n=1}^{\infty}\pi(x^{1/n})/n$. We also mention that, for the sake of convenience, we choose to define the logarithmic integral as 
\begin{equation}
\label{normalized Li eq}
\Li (x):= \int_{1}^{x} \frac{1-u^{-1}}{\log u} \dif u.
\end{equation}

\begin{theorem}
\label{th: Diamond's theorem}
Suppose there exist constants $\alpha\in (0,1]$ and $c>0$, with the additional requirement $c\leq 1$ when $\alpha=1$,  such that \begin{equation}
\label{eq: Diamond sharp PNT}
\Pi(x)= \Li(x)+ O(x\exp (-c \log^{\alpha} x )).
\end{equation}
Then, there is a constant $\rho>0$ such that 
\begin{equation}
\label{eq: Diamond sharp}
	N(x) = \rho x + O\biggl\{x\exp\biggl(-(c(\alpha+1))^{\frac{1}{\alpha+1}}(\log x\log_{2} x)^{\frac{\alpha}{\alpha+1}}\biggl(1+O\biggl(\frac{\log_{3}x}{\log_{2} x}\biggr)\biggr)\biggr)\biggr\}.
\end{equation}
\end{theorem}
A proof of Theorem \ref{th: Diamond's theorem} can be given as in  \cite[Theorem A.1]{B-D-V2020} (cf. \cite{Balazard1999}), starting from the identity \cite{DiamondZhangbook}
\begin{equation}
\label{eq: def number system}
	\dif N = \expm(\dif \Pi) = \sum_{n=0}^{\infty}\frac{1}{n!}(\dif\Pi)\nconv
\end{equation}
and using a version of the Dirichlet hyperbola method to estimate the convolution powers $(\dif\Pi)\nconv$. The current article is devoted to showing the optimality of Theorem \ref{th: Diamond's theorem}, including the optimality of the constant $c' = (c(\alpha+1))^{1/(\alpha+1)}$ in the asymptotic estimate \eqref{eq: Diamond sharp}, as established by the next theorem. Note that Theorem \ref{theorem Malliavin exponent} follows  at once upon combining Theorem \ref{th: Diamond's theorem} and Theorem \ref{th: optimality Diamond's theorem}.
\begin{theorem}
\label{th: optimality Diamond's theorem}

Let $\alpha$ and $c$ be constants such that $\alpha\in (0,1]$ and $c>0$, where we additionally require $c\le 1$ if $\alpha=1$. Then there exists a Beurling generalized number system such that

\begin{equation}
\label{eq: asymptotics Pi}
\Pi(x) - \Li(x) \ll \begin{dcases}												x\exp(-c(\log x)^{\alpha})	&\mbox{if $\alpha<1$ or $\alpha=1$ and $c<1$,} \\
													\log_{2}x 				&\mbox{if $\alpha=c=1$,}
												\end{dcases} 
\end{equation}
and
\begin{equation}
\label{eq: oscillation N}
	N(x) = \rho x + \Omega_{\pm}\biggl\{x\exp\biggl(-(c(\alpha+1))^{\frac{1}{\alpha+1}}(\log x\log_{2} x)^{\frac{\alpha}{\alpha+1}}\biggl(1 + b\frac{\log_{3} x}{\log_{2} x}\biggr)\biggr)\biggr\},
\end{equation}
where $\rho>0$ is the asymptotic density of $N$ and $b$ is some positive constant\footnote{Our example shows that we may select any $b > \alpha/(\alpha+1)$.}.
\end{theorem}

The proof of Theorem \ref{th: optimality Diamond's theorem} consists of two main steps. We shall first construct an explicit example of a continuous analog \cite{Beurling1937,DiamondZhangbook} of a number system fulfilling all requirements from Theorem \ref{th: optimality Diamond's theorem}, and then we will discretize it by means of a probabilistic procedure. The second step will be accomplished in Section \ref{sec: Discretization} with the aid of a recently improved version \cite{B-V2021} of the Diamond-Montgomery-Vorhauer-Zhang random prime approximation method \cite{DiamondMontgomeryVorhauer,Zhang2007}.
The construction and analysis of the continuous example will be carried out in Sections \ref{section: construction continuous example}--\ref{Section conclusion continuous example}. 

Our method is in the same spirit as in \cite{B-D-V2020}, particularly making extensive use of saddle point analysis. Nevertheless, it is worthwhile to point out that showing Theorem \ref{th: optimality Diamond's theorem} requires devising a new example. Even in the case $\alpha=1$ our treatment here delivers novel important information that cannot be reached with the earlier construction. Direct generalizations of the example from \cite{B-D-V2020} are unable to reveal the optimal constant $c'$ in the remainder $O(x\exp(-c' (\log x \log_2 x)^{\alpha/(\alpha+1)}(1+o(1))))$ of \eqref{eq: Diamond sharp}. In fact, upon sharpening the technique from \cite{B-D-V2020} when $\alpha=1$, one would only be able to obtain the $\Omega_{\pm}$-estimate  with $c'>2\sqrt{c}$, which falls short of the actual optimal value $c'=\sqrt{2c}$ that we establish with our new construction. Furthermore, we deal here with the general case $0<\alpha\leq 1$. There is a notable difference between generalized number systems satisfying \eqref{eq: asymptotics Pi} with $\alpha=1$ and those  satisfying it with $0<\alpha<1$. In the latter case, the zeta function admits, in general, no meromorphic continuation\footnote{The asymptotic estimate \eqref{eq: asymptotics Pi} only ensures that, after subtraction of a simple pole-like term, the corresponding zeta function has a boundary value function on $\sigma=1$ that belongs to a non-quasianalytic Gevrey class. } beyond the line $\sigma=1$, which a priori renders direct use of complex analysis arguments impossible. We will overcome this difficulty with a truncation idea, where the analyzed continuous number system is approximated by a sequence of continuous number systems having very regular zeta functions in the sense that they are actually analytic on $\mathbb{C}\setminus\{1\}$.

We conclude this introduction by mentioning that determining the best exponent $\alpha^{\ast}(\beta)$ from Problem \ref{Malliavin open problem} remains wide open for $0<\beta<1$. Bateman and Diamond have conjectured that $\alpha^{\ast}(\beta)=\beta/(\beta+1)$. The validity of this conjecture has only been verified \cite{DiamondMontgomeryVorhauer} for $\beta=1$. It has recently been shown \cite{Broucke2021} that $\alpha^{\ast}(\beta)\leq \beta/(\beta+1)$. However, the best known admissible value \cite[Theorem 16.8, p. 187]{DiamondZhangbook} when $0<\beta<1$ is $\alpha^{*}\approx\beta/(\beta+ 6.91)$, which is still far from the conjectural exponent. 
\section{Construction of the continuous example}
\label{section: construction continuous example}
We explain here the setup for the construction of our continuous example, whose analysis shall be the subject of Sections \ref{sec: The contribution from the saddle points}--\ref{Section conclusion continuous example}. Let us first clarify what is meant by a not necessarily discrete generalized number system. In a broader sense \cite{Beurling1937,DiamondZhangbook}, a Beurling generalized number system is merely a pair of non-decreasing right continuous functions $(\Pi,N)$ with $\Pi(1)=0$ and $N(1)=1$, both having support in $[1,\infty)$, and subject to the relation \eqref{eq: def number system}, where the exponential is taken with respect to  the (multiplicative) convolution of measures \cite{DiamondZhangbook}. Since our hypotheses always guarantee convergence of the Mellin transforms, the latter becomes equivalent to the zeta function identity  
\[
\zeta(s) :=\int^{\infty}_{1^{-}} x^{-s}\mathrm{d}N(x)= \exp\left(\int^{\infty}_{1}x^{-s}\mathrm{d}\Pi(x)\right).
\]

We define our continuous Beurling system 
via its Chebyshev 
function $\psi_{C}$. This uniquely defines $\Pi_{C}$ and $N_{C}$ by means of the relations $\dif\Pi_{C}(u) = (1/\log u)\dif\psi_{C}(u)$ and $\dif N_{C} = \expm(\dif\Pi_{C})$. For $x\ge1$, set 
\begin{equation}
\label{eq: definition psi}
	\psi_{C}(x) = x - 1 - \log x + \sum_{k=0}^{\infty}(R_{k}(x) + S_{k}(x)).
\end{equation}
Here
$x-1-\log x = \int_{1}^{x}\log u \dif\Li(u)$ is the main term (cf. \eqref{normalized Li eq}), the terms
 $R_{k}$ are the deviations which will create a large oscillation in the integers, while the
 $S_{k}$ are introduced to mitigate the jump discontinuity of $R_{k}$ and make $\psi_{C}$ absolutely continuous. The effect of the terms $S_{k}$ on the asymptotics of $N_{C}$ will be harmless. Concretely, we consider fast growing sequences $(A_{k})_{k}$, $(B_{k})_{k}$, $(C_{k})_{k}$, and $(\tau_{k})_{k}$ with $A_{k}<B_{k}<C_{k}<A_{k+1}$, and define\footnote{The factor $1/2$ in the definitions of the functions $R_{k}$ and $S_{k}$ shall be needed to carry out the discretization procedure in the case  $\alpha = 1$ and $c > 1/2$, cf. Lemma \ref{lem: pi_{C}}.}
\begin{align*}
	R_{k}(x)	&= \begin{dcases}
					\mathrlap{\frac{1}{2}\int_{A_{k}}^{x}(1-u^{-1})\cos(\tau_{k}\log u)\dif u}
					\phantom{R_{k}(B_{k}) + \frac{1}{2}\bigl(B_{k}-1-\log B_{k} - (x-1-\log x)\bigr)}	 	&\mbox{for } A_{k}\le x \le B_{k}, \\
					0																&\mbox{otherwise;}
			\end{dcases}\\
	S_{k}(x)	&= \begin{cases}
					R_{k}(B_{k}) + \frac{1}{2}\bigl(B_{k}-1-\log B_{k} - (x-1-\log x)\bigr)		&\mbox{for } B_{k} < x < C_{k},\\
					0														&\mbox{otherwise.}
			\end{cases}	
\end{align*}
We require that $\tau_{k}\log A_{k}, \tau_{k}\log B_{k} \in 2\pi\Z$ and define $C_{k}$ as the unique solution of $R_{k}(B_{k}) + (1/2)\bigl(B_{k}-1-\log B_{k} - (C_{k}-1-\log C_{k})\bigr)=  0$. Notice that for $A_{k}\le x \le B_{k}$,
\begin{align*}	
	R_{k}(x) 		&= \frac{\tau_{k}^{2}}{2(\tau_{k}^{2}+1)}\biggl(\frac{x}{\tau_{k}}\sin(\tau_{k}\log x) + \frac{x}{\tau_{k}^{2}}\cos(\tau_{k} \log x) - \frac{A_{k}}{\tau_{k}^{2}}\biggr) - \frac{\sin(\tau_{k}\log x)}{2\tau_{k}}, \\
	R_{k}(B_{k})	&= \frac{B_{k}-A_{k}}{2(\tau_{k}^{2}+1)} > 0,
\end{align*}
so the definition of $C_{k}$ makes sense (i.e.\ $C_{k}>B_{k}$). We will also set $A_{k}=\sqrt{B_{k}}$ and 
\begin{equation}
\label{eq: def tau}
	\tau_{k} = \exp\bigl(c(\log B_{k})^{\alpha}\bigr),
\end{equation}
then 
\begin{equation}
\label{eq: approx C}
	C_{k} = B_{k}\bigl(1+O(\exp(-2c(\log B_{k})^{\alpha}))\bigr).
\end{equation}
With these definitions in place, we have that $\psi_{C}$ is absolutely continuous, non-decreasing\textcolor{red}{,} and satisfies $\psi_{C}(x) = x + O\bigl(x\exp(-c(\log x)^{\alpha})\bigr)$, which implies that\footnote{When $\alpha=c=1$, the stronger asymptotic estimate $\Pi_{C}(x)=\Li(x)+O(1)$ holds.} \eqref{eq: asymptotics Pi} holds for $\Pi_{C}(x) = \int_{1}^{x}(1/\log u)\dif\psi_{C}(u)$. Finally we define a sequence $(x_{k})_{k}$ via the relation
\begin{equation}
\label{eq: def x}
	\log B_{k} = (c(\alpha+1))^{\frac{-1}{\alpha+1}}(\log x_{k}\log_{2} x_{k})^{\frac{1}{\alpha+1}} + \eps_{k}.
\end{equation}
Here $(\eps_{k})_{k}$ is a bounded sequence which is introduced to control the value of $\tau_{k}\log x_{k}$ mod $2\pi$ (this will be needed later on).
It is on the sequence $(x_{k})_{k}$ that we will show the oscillation estimate \eqref{eq: oscillation N}.

We collect all technical requirements of the considered sequences in the following lemma. The rapid growth of the sequence $(B_{k})_{k}$ will be formulated as a general inequality $B_{k+1}>\max\{F(B_{k}), G(k)\}$, for some functions $F$ and $G$. We will not specify here what $F$ and $G$ we require. At each point later on where the rapid growth is used, it will be clear what kind of growth (and what $F$, $G$) is needed.
\begin{lemma}
Let $F$, $G$ be increasing functions. There exist sequences $(B_{k})_{k}$ and $(\eps_{k})_{k}$ such that, with the definitions of $(A_{k})_{k}$, $(C_{k})_{k}$, $(\tau_{k})_{k}$, and $(x_{k})_{k}$ as above, the following properties hold:
\begin{enumerate}[label = (\alph*)]
	\item $B_{k+1} > \max\{F(B_{k}), G(k)\}$; \label{Property (a)}
	\item $\tau_{k}\log A_{k}\in 2\pi\Z$ and $\tau_{k}\log B_{k} \in 2\pi\Z$; \label{Property (b)}
	\item $\tau_{k}\log x_{k} \in \pi/2 + 2\pi\Z$ when $k$ is even, and $\tau_{k}\log x_{k} \in 3\pi/2 + 2\pi\Z$ when $k$ is odd;\label{Property (c)}
	\item $(\eps_{k})_{k}$ is a bounded sequence. \label{Property (d)}
\end{enumerate}
\end{lemma} 
\begin{proof}
We define the sequences inductively. Consider the function $f(u) = u\e^{cu^{\alpha}}$. 
Let $B_{0}$ be some (large) number with $f(\log B_{0})\in4\pi\Z$, so that \ref{Property (b)} is satisfied with $k=0$. Define $y_{0}$ via $\log B_{0} = (c(\alpha+1))^{\frac{-1}{\alpha+1}}(\log y_{0}\log_{2} y_{0})^{\frac{1}{\alpha+1}}$. We have that $\tau_{0}\log x_{0} - \tau_{0}\log y_{0} \asymp -\eps_{0}\tau_{0}(\log B_{0})^{\alpha}/\log_{2}B_{0}$, if $\eps_{0}$ is bounded, say, so we may pick an $\eps_{0}$ satisfying even $0 \le \eps_{0} \ll \tau_{0}^{-1}(\log B_{0})^{-\alpha}\log_{2}B_{0}$ so that $\tau_{0}\log x_{0} \in \pi/2+2\pi\Z$.

Now suppose that $B_{k}$ and $\eps_{k}$, $0\le k \le K$ are defined. Choose a number $B_{K+1}>\max\{4(C_{K})^{2}, F(B_{K}), G(k)\}$ with $f(\log B_{K+1}) \in 4\pi\Z$, taking care of \ref{Property (a)} and \ref{Property (b)}. As before, one might choose $\eps_{K+1}$, $0\le \eps_{K+1}\ll \tau_{K+1}^{-1}(\log B_{K+1})^{-\alpha}\log_{2}B_{K+1}$ such that \ref{Property (c)} holds. Property \ref{Property (d)} is obvious. 
\end{proof}

In order to deduce the asymptotics of $N_{C}$, we shall analyze its zeta function $\zeta_{C}$ and use
an effective Perron formula:
\begin{equation}
\label{eq: Perron inversion}
	N_{C}(x) = \frac{1}{2\pi\I}\int_{\kappa-\I T}^{\kappa+\I T}x^{s}\zeta_{C}(s)\frac{\dif s}{s} + \mbox{ error term}.
\end{equation}
Here $\kappa>1$, the parameter $T>0$ is some large number, and the error term depends on these numbers. The usual strategy is then to push the contour of integration to the left of $\sigma = \Re s =1$; the pole of $\zeta_{C}$ at $s=1$ will give the main term, while lower order terms will arise from the integral over the new contour (whose shape will be dictated by the growth of $\zeta_{C}$). In its current form, this approach is not suited for our problem,
 since it is not clear if our zeta function admits a meromorphic continuation to the left of $\sigma=1$. However, we can remedy this with the following truncation idea.

Consider $x\ge 1$ and let $K$ be such that $x < A_{K+1}$. We denote by $\psi_{C,K}$ the Chebyshev function defined by \eqref{eq: definition psi}, but where the summation range in the series is altered to the restricted range $ 0\le k\le K$. For $x<A_{K+1}$ we have $\psi_{C,K}(x) = \psi_{C}(x)$, and, setting $\dif \Pi_{C,K}(u) = (1/\log u) \dif\psi_{C,K}(u)$ and $\dif N_{C,K}(u) = \expm(\dif\Pi_{C,K}(u))$, we also have that $N_{C,K}(x) = N_{C}(x)$ holds in this range. Hence for these $x$, the above Perron formula \eqref{eq: Perron inversion} remains valid if we replace $\zeta_{C}$ by $\zeta_{C,K}$, the zeta function of $N_{C,K}$, which does admit meromorphic continuation beyond $\sigma=1$. 

In the following two sections, we will study the Perron integral in \eqref{eq: Perron inversion} for $x = x_{K}$ and with $\zeta_{C}$ replaced 
by $\zeta_{C,K}$. Note that by \ref{Property (a)}, we may assume that $x_{K} < A_{K+1}$. To asymptotically evaluate this integral, we will use the saddle point method, also known as the method of steepest descent. For an introduction to the saddle point method, we refer to \cite[Chapters 5 and 6]{deBruijn} or \cite[Section 3.6]{EstradaKanwalbook}. 

In Section \ref{sec: The contribution from the saddle points} we will estimate the contribution from the integral over the steepest paths through the saddle points. This contribution will match the oscillation term in \eqref{eq: oscillation N}.
 In Section \ref{sec: The remainder}, we will connect these steepest paths to each other and to the vertical line $[\kappa-\I T, \kappa+\I T]$ and determine that the contribution of these connecting pieces to \eqref{eq: Perron inversion} is of lower order than the contribution from the saddle points. We also estimate the error term in the effective Perron formula in Section \ref{Section conclusion continuous example}, and conclude the analysis of the continuous example. Finally, in Section \ref{sec: Discretization} we use probabilistic methods to show the existence of a \emph{discrete} Beurling system $(\Pi, N)$ that inherits the asymptotics of the continuous system $(\Pi_{C}, N_{C})$.

\section{Analysis of the saddle points}
\label{sec: The contribution from the saddle points}
First we compute the zeta function $\zeta_{C,K}$. Computing the Mellin transform of $\psi_{C,K}$ gives that
\[
	-\frac{\zeta_{C,K}'}{\zeta_{C,K}}(s) = \frac{1}{s-1} - \frac{1}{s} + \sum_{k=0}^{K}\bigl( \eta_{k}(s) + \tilde{\eta}_{k}(s) + \xi_{k}(s) -  \eta_{k}(s+1) - \tilde{\eta}_{k}(s+1) - \xi_{k}(s+1)\bigr),
\]
where 
\begin{equation}
\label{eq: def eta}
	\eta_{k}(s) = \frac{B_{k}^{1-s} - A_{k}^{1-s}}{4(1+\I\tau_{k} -s)}, \quad \tilde{\eta}_{k}(s) =  \frac{B_{k}^{1-s} - A_{k}^{1-s}}{4(1-\I\tau_{k} -s)}, \quad 
	\xi_{k}(s) = \frac{B_{k}^{1-s}-C_{k}^{1-s}}{2(1-s)},
\end{equation}
and where we used property \ref{Property (b)} of the sequences $(A_{k})_{k}$, $(B_{k})_{k}$.
Integrating gives 
\[
	\log \zeta_{C,K}(s) = \log\frac{s}{s-1} + \sum_{k=0}^{K}\int_{s}^{s+1}\bigl(\eta_{k}(z) + \tilde{\eta}_{k}(z) + \xi_{k}(z)\bigr)\dif z,
\]
the integration constant being $0$ because $\log \zeta_{C,K}(\sigma) \rightarrow 0$ 
as $\sigma \rightarrow \infty$. The main term of the Perron integral formula for $N_{C,K}(x_{K})$ becomes 
\[
	\frac{1}{2\pi\I}\int_{\kappa-\I T}^{\kappa+\I T}\frac{x_{K}^{s}}{s-1}\exp\biggl(\sum_{k=0}^{K}\int_{s}^{s+1}\bigl(\eta_{k}(z) + \tilde{\eta}_{k}(z) + \xi_{k}(z)\bigr)\dif z\biggr)\dif s.
\]
The idea of the saddle point method is to estimate an integral of the form $\int_{\Gamma}\e^{f(s)}g(s)\dif s$, with $f$ and $g$ analytic,
by shifting the contour $\Gamma$ to a contour which passes through the saddle points of $f$ via the paths of steepest descent. 
Since the main contribution in the Perron integral will come from $x_{K}^{s}\exp(\int_{s}^{\infty}\eta_{K}(z)\dif z)$, we will apply the method with
\begin{align}
	\label{eq: def f}
	f(s) 	&= f_{K}(s) = s\log x_{K} + \int_{s}^{\infty}\eta_{K}(z)\dif z, \\
	\label{eq: def g}
	g(s) 	&= g_{K}(s) = \frac{1}{s-1}\exp\biggl(\sum_{k=0}^{K}\int_{s}^{s+1}\bigl(\eta_{k}(z) + \tilde{\eta}_{k}(z) + \xi_{k}(z)\bigr)\dif z - \int_{s}^{\infty}\eta_{K}(z)\dif z\biggr).
\end{align}
Note also that by writing $\int_{s}^{\infty}\eta_{K}(z)\dif z $ as a Mellin transform, we obtain the alternative representation
\begin{equation}
\label{eq: formula int eta}
	\int_{s}^{\infty}\eta_{K}(z)\dif z = \frac{1}{4}\int_{A_{K}}^{B_{K}}x^{-s}\e^{\I\tau\log x}\frac{1}{\log x}\dif x = \frac{1}{4}\int_{1/2}^{1}\frac{B_{K}^{(1+\I\tau_{K} - s)u}}{u}\dif u,
\end{equation}
as we have set $A_{K} = \sqrt{B_{K}}$.
In the rest of this section, we will mostly work with $f_{K}$, and we will drop the subscripts $K$ where
 there is no risk of confusion.

\subsection{The saddle points}
We will now compute the saddle points of $f$, which are solutions of the equation
\begin{equation}
\label{eq: saddle point}
	f'(s) = \log x - \frac{1}{4}B^{1-s}\frac{1- B^{(s-1)/2}}{1+\I\tau - s} = 0.
\end{equation}
For integers $m$, set numbers $t^{\pm}_{m}$ as $t^{\pm}_{m} = \tau + (2\pi m \pm \pi/2)/\log B$, and let $V_{m}$ be the rectangle with vertices 
\[
	1-\frac{\frac{\alpha}{2}\log_{2} B}{\log B} + \I t^{\pm}_{m}, \quad \frac{1}{2} + \I t^{\pm}_{m}. 
\]	
\begin{lemma}
	Suppose that $\abs{m} < \log_{2} B$. Then $f'$ has a unique simple zero $s_{m}$ in the interior of $V_{m}$.
\end{lemma}
\begin{proof}
We apply the argument principle. Note that from \eqref{eq: def x} it follows that
\begin{align*}
	f'\biggl(\frac{1}{2} + \I t^{-}_{m}\biggr) 	&= -\frac{\I}{2}B^{1/2}\bigl(1+o(1)\bigr), & f'\biggl(1-\frac{\frac{\alpha}{2}\log_{2} B}{\log B} + \I t^{-}_{m}\biggr)	&=\log x\bigl(1+o(1)\bigr), \\
	f'\biggl(1-\frac{\frac{\alpha}{2}\log_{2} B}{\log B} + \I t^{+}_{m}\biggr)	&=\log x\bigl(1+o(1)\bigr), & f'\biggl(\frac{1}{2} + \I t^{+}_{m}\biggr)	&= \frac{\I}{2}B^{1/2}\bigl(1+o(1)\bigr).
\end{align*}
On the lower horizontal side of $V_{m}$, we have
\[
	\Im f'(\sigma+\I t^{-}_{m}) = -\frac{B^{1-\sigma}/4}{(1-\sigma)^{2}+(\tau-t_{m}^{-})^{2}}\biggl\{ \biggl(1-\frac{\sqrt{2}}{2}B^{\frac{\sigma-1}{2}}\biggr)(1-\sigma) 
	+ \frac{\sqrt{2}}{2}B^{\frac{\sigma-1}{2}}(\tau-t_{m}^{-})\biggr\} < 0,
\]
as the factor inside the curly brackets is positive in the considered ranges for $\sigma$ and $m$. Similarly we have $\Im f'(\sigma+\I t^{+}_{m})>0$ on the upper horizontal edge of $V_{m}$. On the right vertical edge, 
\[
	\Re f'\biggl(1-\frac{\frac{\alpha}{2}\log_{2}B}{\log B} +\I t\biggr) >0,
\]
and on the left vertical edge, 
\[
	f'\biggl(\frac{1}{2}+\I t\biggr) = \frac{B^{1/2}}{2}\e^{\I\pi-\I(t-\tau)\log B}\bigl(1+o(1)\bigr).
\]

Starting from the lower left vertex of $V_{m}$ and moving in the counterclockwise direction, we see that the argument of $f'$ starts off close to $-\pi/2$, increases to about $0$ on the lower horizontal edge, remains close to $0$ on the right vertical edge, increases to about $\pi/2$ on the upper horizontal edge, and finally increases to approximately $3\pi/2$ on the left vertical edge. This proves the lemma.
\end{proof}
From now on, we assume that $\abs{m} < \eps \log_{2} B$ for some small $\eps>0$. (In fact, later on we will further reduce the range 
 to $\abs{m}\le (\log_{2} B)^{3/4}$.) We denote the unique saddle point in the rectangle $V_{m}$ by $s_{m} = \sigma_{m} + \I t_{m}$.
The saddle point equation \eqref{eq: saddle point} implies that
\begin{align*}
	\sigma_{m}	&= 1-\frac{1}{\log B}\biggl(\log_{2} x + \log 4 - \log\abs[1]{1-B^{(s_{m}-1)/2}} - \log\abs{\frac{1}{1+\I\tau-s_{m}}}\,\biggr), \\
	t_{m}			&= \tau + \frac{1}{\log B}\biggl(2\pi m + \arg\bigl(1-B^{(s_{m}-1)/2}\bigr) - \arg\bigl(1+\I\tau - s_{m}\bigr)\biggr),
\end{align*}
with the understanding that the difference of the arguments in the formula for $t_{m}$ lies in $[-\pi/2,\pi/2]$.
We set 
\[
	E_{m} = \log\abs{\frac{1}{1+\I\tau - s_{m}}}.
\]
Since $s_{m}\in V_{m}$, we have $0\le E_{m}\le \log_{2} B$. Also $\log\abs[1]{1-B^{(s_{m}-1)/2}} = O(1)$. This implies that 
\[
	\sigma_{m} = 1 - \frac{1}{\log B}\bigl(\log_{2} x - E_{m} + O(1)\bigr),
\]
so that $E_{m} = \log_{2} B - \log_{3}x + O(1)$. Here we have also used that 
\[
	\tau-t_{m} \ll \frac{\log_{2} B}{\log B}, \quad \mbox{and} \quad \log_{2} B \sim \frac{1}{\alpha+1}\log_{2} x,
\]
the last formula following from \eqref{eq: def x}. This in turn implies that 
\begin{equation}
\label{eq: approx sigma}
	\sigma_{m} = 1-\frac{\alpha\log_{2} B + O(1)}{\log B}, 
\end{equation}
where we again used \eqref{eq: def x}. Combining this with \eqref{eq: saddle point} we get in particular that
\begin{equation}
\label{eq: approx log x}
	\log x = \frac{B^{1-s_{m}}}{4(1+\I\tau-s_{m})}\bigl(1+O\bigl((\log B)^{-\alpha/2}\bigr)\bigr).
\end{equation}
For $t_{m}$, we have that 
\begin{align*}
	\arg(1-B^{(s_{m}-1)/2}) 	&\ll (\log B)^{-\alpha/2}, \\
	\arg(1+\I\tau - s_{m})	&= -\frac{2\pi m}{\alpha\log_{2} B} + O\biggl(\frac{1}{\log_{2} B} + \frac{\abs{m}}{(\log_{2} B)^{2}} + \frac{\abs{m}^{3}}{(\log_{2} B)^{3}}\biggr).
\end{align*}
We get that 
\begin{equation}
\label{eq: approx t}
	t_{m} = \tau + \frac{1}{\log B}\biggl\{2\pi m\biggl(1+\frac{1}{\alpha\log_{2} B}\biggr) + O\biggl(\frac{1}{\log_{2} B} + \frac{\abs{m}}{(\log_{2} B)^{2}} + \frac{\abs{m}^{3}}{(\log_{2} B)^{3}}\biggr) \biggr\}.
\end{equation}
Also, it is important to notice that  $t_{0} = \tau$.

The main contribution to the Perron integral \eqref{eq: Perron inversion} will come from the saddle point $s_{0}$; see Subsection \ref{subsection: contribution from s_0}. We will show in Subsection \ref{subsection: contributions from s_m} that the contribution from the other saddle points $s_{m}$, $m\neq 0$, is of lower order. This will require a finer estimate for $\sigma_{m}$, which is the subject of the following lemma.
\begin{lemma}
\label{lem: finer estimate sigma}
There exists a fixed constant $d>0$, independent of $K$ and $m$, such that for $\abs{m}\le (\log_{2}B)^{3/4}$, $m\neq0$, 
\[
	\sigma_{m} \le \sigma_{0} - \frac{d}{\log B(\log_{2} B)^{2}}.
\]
\end{lemma} 
\begin{proof}
We use \eqref{eq: approx sigma} and \eqref{eq: approx t} to get a better estimate for $E_{m}$, which will in turn yields a better estimate for $\sigma_{m}$. We iterate this procedure three times. 

The first iteration yields 
\[
	\sigma_{m} = 1 - \frac{1}{\log B}\biggl\{\log_{2} x - \log_{2} B + \log_{3} B + \log 4 + \log\alpha + O\biggl(\frac{1+\abs{m}}{\log_{2} B}\biggr)\biggr\}.
\]
Write $Y = \log_{2}x - \log_{2}B + \log_{3}B$ and note that $Y \asymp \log_{2}B$. Iterating a second time, we get 
\[
	\sigma_{m} = 1-\frac{1}{\log B}\biggl\{\log_{2}x - \log_{2}B + \log Y +\log 4 + \frac{\log4+\log\alpha}{Y} + O\biggl(\frac{1+m^{2}}{(\log_{2}B)^{2}}\biggr)\biggr\}.
\]
We now set $Y' = \log_{2}x-\log_{2}B + \log Y$, and note again that $Y'\asymp \log_{2}B$. A final iteration gives
\begin{multline*}
	\sigma_{m} = 1 - \frac{1}{\log B}\biggl\{\log_{2}x - \log_{2}B + \log Y' + \log4 + \frac{\log 4}{Y'} + \frac{\log4+\log\alpha}{YY'} \\ 
	-\frac{(\log4)^{2}}{2Y'^{2}}
	+\frac{2\pi^{2}m^{2}}{Y'^{2}} - \frac{4\pi^{4}m^{4}}{Y'^{4}} + O\biggl(\frac{1+m^{2}}{(\log_{2}B)^{3}}\biggr)\biggr\}
\end{multline*}
The lemma now follows from comparing the above formula in the case $m=0$ with the case $m\neq0$.
\end{proof}
Near the saddle points we will approximate $f$ and $f'$ by their Taylor polynomials.
\begin{lemma}
\label{lem: approx f}
There are holomorphic functions $\lambda_{m}$ and $\tilde{\lambda}_{m}$ such that 
\begin{align*}
	f(s)	&= f(s_{m}) + \frac{f''(s_{m})}{2}(s-s_{m})^{2}(1+\lambda_{m}(s)), \\
	f'(s)	&= f''(s_{m})(s-s_{m})(1+\tilde{\lambda}_{m}(s)),
\end{align*}
and with the property that for each $\eps>0$ there exists a $\delta>0$, independent of $K$ and $m$, such that 
\[
	\abs{s-s_{m}} < \frac{\delta}{\log B} \implies \abs{\lambda_{m}(s)}+\abs[1]{\tilde{\lambda}_{m}(s)} < \eps.
\]
\end{lemma}
\begin{proof}
We have 
\[
	f''(s) = (\log B)\frac{B^{1-s} - \frac{1}{2} B^{(1-s)/2}}{4(1+\I\tau - s)} - \frac{B^{1-s} - B^{(1-s)/2}}{4(1+\I\tau-s)^{2}}, 
	\quad \abs{f''(s_{m})} \asymp \frac{(\log B)^{\alpha}(\log B)^{2}}{\log_{2}B},
\]
where we have used \eqref{eq: approx sigma},
and
\[
	f'''(s) = -(\log B)^{2}\frac{B^{1-s} - \frac{1}{4}B^{(1-s)/2}}{4(1+\I\tau-s)} + (\log B)\frac{B^{1-s}-\frac{1}{2}B^{(1-s)/2}}{2(1+\I\tau-s)^{2}} - \frac{B^{1-s}-B^{(1-s)/2}}{2(1+\I\tau-s)^{3}}.
\]
If $\abs{s-s_{m}}\ll 1/\log B$, then 
\[
	\abs[1]{f'''(s)} \ll \frac{(\log B)^{\alpha}(\log B)^{3}}{\log_{2}B}.
\]
It follows that 
\[
	\abs[3]{\frac{f'''(s)}{f''(s_{m})}(s-s_{m})} < \eps,
\]
if $\abs{s-s_{m}} < \delta/\log B$, for sufficiently small $\delta$. The lemma now follows from Taylor's formula. 
\end{proof}

\subsection{The steepest path through $s_{0}$}

The equation for the path of steepest descent through $s_{0}$ is 
\[
	\Im f(s) = \Im f(s_{0}) \mbox{ under the constraint } \Re f(s) \le \Re f(s_{0}).
\] 
Using the formula \eqref{eq: formula int eta} for  $\int_{s}^{\infty}\eta(z)\dif z$, we get the equation
\[
	t\log x - \frac{1}{4}\int_{1/2}^{1}B^{(1-\sigma)u}\sin\bigl((t-\tau)(\log B) u\bigr)\frac{\dif u}{u} = \tau\log x.
\]
Setting $\theta=(t-\tau)\log B$, this is equivalent to
\begin{equation}
\label{eq: steepest path s_{0}}
	\theta\frac{\log x}{\log B} = \frac{1}{4}\int_{1/2}^{1}B^{(1-\sigma)u}\sin(\theta u)\frac{\dif u}{u}.
\end{equation}
Note that, as $t$ varies between $t_{0}^{-}$ and $t_{0}^{+}$, $\theta$ varies between $-\pi/2$ and $\pi/2$. This equation has every point of the line $\theta=0$ as a solution. However, one sees that the line $\theta=0$ is the path of steepest \emph{ascent}, since $\Re f(s) \ge \Re f(s_{0})$ there. We now show the existence of a different curve through $s_{0}$ of which each point is a solution of \eqref{eq: steepest path s_{0}}. This is then necessarily the path of steepest \emph{descent}. For each fixed $\theta\in [-\pi/2, \pi/2] \setminus \{0\}$, equation \eqref{eq: steepest path s_{0}} has a unique solution $\sigma=\sigma_{\theta}$, since the right hand side is a continuous and monotone function of $\sigma$, with range $\R_{\gtrless 0}$, if $\theta\gtrless 0$. This shows the existence of the path of steepest descent $\Gamma_{0}$ through $s_{0}$. This path connects the lines $\theta=-\pi/2$ and $\theta=\pi/2$.

One can readily see that 
\[
	\sigma_{\theta} = \sigma_{0} - \frac{a_{\theta}}{\log B}, \quad \mbox{where } \abs{a_{\theta}}\ll 1.
\]
Integrating by parts, we see that 
\begin{align*}
	\frac{1}{4}\int_{1/2}^{1}B^{(1-\sigma_{\theta})u}\sin(\theta u)\frac{\dif u}{u} 
		&= \frac{1}{4}\sin\theta \frac{B^{1-\sigma_{\theta}}}{(1-\sigma_{\theta})\log B}\bigl(1+O\bigl((\log B)^{-\alpha/2}\bigr) + O\bigl((\log_{2} B)^{-1}\bigr)\bigr) \\
		&=\frac{\sin\theta}{4\log B}\frac{B^{1-\sigma_{0}}}{1-\sigma_{0}}\e^{a_{\theta}}\bigl(1+O\bigl((\log_{2} B)^{-1}\bigr)\bigr) \\
		&=\sin\theta\frac{\log x}{\log B}\e^{a_{\theta}}\bigl(1+O\bigl((\log_{2} B)^{-1}\bigr)\bigr),
\end{align*}
where we used \eqref{eq: approx log x} in the last line. Equation \eqref{eq: steepest path s_{0}} then implies that 
\begin{equation}
\label{eq: a_{theta}}
	\e^{a_{\theta}} = \frac{\theta}{\sin\theta} + O\bigl((\log_{2} B)^{-1}\bigr).
\end{equation}

Let $\gamma$ now be a unit speed parametrization of this path of steepest descent:
\[
	\gamma: [y^{-}, y^{+}] \to \Gamma_{0}, \quad \Im\gamma(y^{-}) = \tau-\frac{\pi/2}{\log B}, \quad \gamma(0) = s_{0}, \quad \Im \gamma(y^{+}) = \tau+\frac{\pi/2}{\log B}, \quad \abs{\gamma'(y)}=1.
\]
The fact that $\Gamma_{0}$ is the path of steepest descent implies that for $y<0$, $\gamma'(y)$ is a positive multiple of $\overline{f'}(\gamma(y))$, while for $y>0$, $\gamma'(y)$ is a negative multiple of $\overline{f'}(\gamma(y))$. We now show that the argument of the tangent vector $\gamma'(y)$ is sufficiently close to $\pi/2$.
\begin{lemma}
\label{lem: bound arg gamma}
For $y\in [y^{-}, y^{+}]$, $\abs[1]{\arg\bigl(\gamma'(y)\e^{-\I\pi/2}\bigr)} < \pi/5$.
\end{lemma}
\begin{proof}
We consider two cases: the case where $s$ is sufficiently close to $s_{0}$ so that we can apply Lemma \ref{lem: approx f} to estimate the argument of $\overline{f'}$, and the remaining case, where we will estimate this argument via the definition of $f$.

We apply Lemma \ref{lem: approx f} with $\eps = 1/5$ to find a $\delta>0$ such that for $\abs{s-s_{0}} < \delta/\log B$, 
\[
	h(s) \coloneqq f(s)-f(s_{0}) = \frac{f''(s_{0})}{2}(s-s_{0})^{2}(1+\lambda_{0}(s)), \quad \abs{\lambda_{0}(s)} < \frac{1}{5}.
\]
Set $s-s_{0} = r\e^{\I\phi}$ with $r<\delta/\log B$ and $-\pi < \phi \le \pi$. Using that $f''(s_{0})$ is real and positive, we have
\begin{align*}
	\Re h(s)	&= \frac{f''(s_{0})}{2}r^{2}\bigl((1+\Re\lambda_{0}(s))\cos2\phi - (\Im\lambda_{0}(s))\sin2\phi\bigr)\\
	\Im h(s)	&= \frac{f''(s_{0})}{2}r^{2}\bigl((1+\Re\lambda_{0}(s))\sin2\phi + (\Im\lambda_{0}(s))\cos2\phi\bigr).
\end{align*}
Suppose $s\in \Gamma_{0}\setminus\{s_{0}\}$ with $\abs{s-s_{0}} < \delta/\log B$. Then $\Re h(s)<0$ and $\Im h(s)=0$. The condition $\Re h(s) < 0$ implies that $\phi\in (-4\pi/5, -\pi/5) \cup (\pi/5, 4\pi/5)$ say, as $\abs{\lambda_{0}(s)} < 1/5$. In combination with $\Im h(s) = 0$ this implies that $\phi\in (-3\pi/5,-2\pi/5) \cup (2\pi/5, 3\pi/5)$ whenever $s\in\Gamma_{0}\setminus\{s_{0}\}$, $\abs{s-s_{0}} < \delta/\log B$. Again by Lemma \ref{lem: approx f}, 
\[
	f'(s) = f''(s_{0})r\e^{\I\phi}(1+\tilde{\lambda}_{0}(s)), \quad \abs[1]{\tilde{\lambda}_{0}(s)} < \frac{1}{5}.
\]
It follows that $\abs[1]{\arg\bigl(\gamma'(y)\e^{-\I\pi/2}\bigr)} < \pi/5$ when $\abs{\gamma(y)-s_{0}} < \delta/\log B$.

It remains to treat the case $\abs{\gamma(y)-s_{0}} \ge \delta/\log B$. For these points, we have that $\delta/2 \le \abs{\theta} \le \pi/2$, where we used the notation $\theta = (\Im \gamma(y) -\tau)\log B$ as before. 
Set $\gamma(y) = s = \sigma+\I t$ with $\sigma=\sigma_{0} - a_{\theta}/\log B$. 
Recalling that $\tau \log B\in 4\pi \mathbb{Z}$, we obtain the following explicit expression for $\overline{f'}$:
\begin{multline*}
	\overline{f'}(s) = \log x - \frac{1/4}{(1-\sigma)^{2}+(t-\tau)^{2}}\biggl\{
	B^{1-\sigma}\biggl(\biggl((1-\sigma)\cos\theta + \frac{\theta\sin\theta}{\log B}\biggr) + \I\biggl((1-\sigma)\sin\theta - \frac{\theta\cos\theta}{\log B}\biggr)\biggr) \\
	-B^{(1-\sigma)/2}\biggl(\biggl((1-\sigma)\cos(\theta/2) + \frac{\theta\sin(\theta/2)}{\log B}\biggr) + \I\biggl((1-\sigma)\sin(\theta/2) - \frac{\theta\cos(\theta/2)}{\log B}\biggr)\biggr)\biggr\}.
\end{multline*}
Using \eqref{eq: approx log x} and \eqref{eq: a_{theta}}, we see that 
\begin{align*}
	\Im \overline{f'}(s) 	&= -\log x\bigl(\theta+O\bigl((\log_{2} B)^{-1}\bigr)\bigr)\textcolor{blue}{,}\\
	\Re \overline{f'}(s)	&= \log x\bigl(1 - \theta\cot\theta + O\bigl((\log_{2} B)^{-1}\bigr)\bigr). 
\end{align*}
This implies
\[
	\abs[1]{\arg\bigl(\gamma'(y)\e^{-\I\pi/2}\bigr)} = \abs[1]{\arctan\bigl(1/\theta - \cot\theta + O_{\delta}\bigl((\log_{2} B)^{-1}\bigr)\bigr)} < \frac{\pi}{5}.
\]
The last inequality follows from the fact that $\abs{1/\theta-\cot\theta} < 2/\pi$ for $\theta\in [-\pi/2, \pi/2]$, and that $\arctan(2/\pi) \approx 0.18\pi < \pi/5$.
\end{proof}

\subsection{The contribution from $s_{0}$}
\label{subsection: contribution from s_0}
We will now estimate the contribution from $s_{0}$, by which we mean
\[
	\frac{1}{\pi}\Im\int_{\Gamma_{0}}\e^{f(s)}g(s)\dif s,
\]
and where $f$ and $g$ are given by \eqref{eq: def f} and \eqref{eq: def g} respectively. We have combined the two pieces in the upper and lower half plane $\int_{\Gamma_{0}}$ and $-\int_{\overline{\Gamma_{0}}}$ into one integral using $\zeta_{C}(\overline{s}) = \overline{\zeta_{C}(s)}$. 
To estimate this
integral, we will use the following simple lemma
(see e.g. \cite[Lemma 3.3]{B-D-V2020}).
\begin{lemma}
\label{lem: lower bound integral}
Let $a<b$ and suppose that $F: [a,b] \to \C$ is integrable. If there exist $\theta_{0}$ and $\omega$ with $0\le \omega < \pi/2$ such that $\abs{\arg(F\e^{-\I\theta_{0}})} \le \omega$, then
\[
	\int_{a}^{b}F(u)\dif u = \rho\e^{\I(\theta_{0}+\vphi)}
\]
for some real numbers $\rho$ and $\vphi$ satisfying
\[
	\rho \ge (\cos \omega)\int_{a}^{b}\abs{F(u)}\dif u \quad \mbox{and} \quad \abs{\vphi} \le \omega.
\]
\end{lemma}

We will estimate $g$ with the following lemma.
\begin{lemma}
\label{lem: bound g}
Let $\eps>0$ and suppose that $s=\sigma+\I t$ satisfies 
\[
	\sigma \ge 1 - O\left(\frac{\log_{2}B_{K}}{\log B_{K}} \right), \quad t \gg \tau_{K},
\]
Then for $K (> K(\eps))$ sufficiently large,
\[
	\abs{\sum_{k=0}^{K-1}\int_{s}^{s+1}\bigl(\eta_{k}(z) + \tilde{\eta}_{k}(z) + \xi_{k}(z)\bigr)\dif z + \int_{s}^{s+1}\bigl(\tilde{\eta}_{K}(z)+\xi_{K}(z)\bigr)\dif z - \int_{s+1}^{\infty}\eta_{K}(z)\dif z} < \eps. 
\]
\end{lemma}
\begin{proof}
By the definition \eqref{eq: def eta} of the functions $\eta_{k}$, $\tilde{\eta}_{k}$, and $\xi_{k}$, we have
\[
	\sum_{k=0}^{K}\int_{s}^{s+1}\xi_{k}(z)\dif z \ll \sum_{k=0}^{K}\frac{C_{k}^{1-\sigma}}{\abs{s}\log C_{k}} \ll K\frac{(\log
	 B_{K})^{O(1)}}{\tau_{K}},
\]
where in the last step we used that $C_{K} \asymp B_{K}$ by \eqref{eq: approx C}. This quantity is bounded by $\exp\bigl(\log K - c(\log B_{K})^{\alpha} + O(\log_{2}B_{K})\bigr)$, which can be made arbitrarily small by taking $K$ sufficiently large, due to the rapid growth of $(B_{k})_{k}$ (property \ref{Property (a)}). The condition $t \gg \tau_{K}$ together with the rapid growth of $(\tau_{k})_{k}$ implies that $\abs{1
\pm
\I\tau_{k} - s} \gg \tau_{K}$, for $0\le k \le K-1$ (at least when $K$ is sufficiently large). Hence,
\[
	\sum_{k=0}^{K-1}\int_{s}^{s+1}(\eta_{k}(z)+\tilde{\eta}_{k}(z))\dif z \ll \sum_{k=0}^{K-1}\frac{B_{k}^{1-\sigma}}{\tau_{K}\log B_{k}} \ll \exp\bigl(\log K - c(\log B_{K})^{\alpha} + O(\log_{2}B_{K})\bigr).
\]
Finally we have
\begin{align*}
	\int_{s}^{s+1}\tilde{\eta}_{K}(z)\dif z 	&\ll \frac{B_{K}^{1-\sigma}}{\tau_{K}\log B_{K}} = \exp\bigl(-c(\log B_{K})^{\alpha} + O(\log_{2}B_{K})\bigr), \\
	\int_{s+1}^{\infty}\eta_{K}(z)\dif z 	&\ll \frac{B_{K}^{-\sigma}}{\log B_{K}}. \qedhere
\end{align*}
\end{proof}
In particular we may assume that on the contour $\Gamma_{0}$, these terms are in absolute value smaller than $\pi/40$\textcolor{blue}{,} say. Also, $1/\abs{s-1} \sim 1/\tau_{K}$ and $\abs[1]{\arg\bigl(\e^{\I\pi/2}/(s-1)\bigr)} < \pi/40$ on $\Gamma_{0}$. We have
\[
	\int_{\Gamma_{0}}\e^{f(s)}g(s)\dif s = \e^{f(s_{0})}\int_{\Gamma_{0}}\e^{f(s)-f(s_{0})}g(s)\dif s.
\]
We now apply Lemma \ref{lem: lower bound integral} to estimate the size and argument of this integral. By Property \ref{Property (c)} and Lemma \ref{lem: bound arg gamma} we get that
\begin{gather*}
	\int_{\Gamma_{0}}\e^{f(s)}g(s)\dif s = (-1)^{K}R\e^{\I(\pi/2 + \vphi)}, \\
	R \gg \frac{\e^{\Re f(s_{0})}}{\tau_{K}}\int_{y^{-}}^{y^{+}}\exp\bigl( f(\gamma(y))-f(s_{0})\bigr)\dif y, \quad \abs{\vphi} < \frac{\pi}{5} + \frac{\pi}{40} + \frac{\pi}{40} = \frac{\pi}{4}.
\end{gather*}
Note that $f(\gamma(y))-f(s_{0})$ is real. In order to bound
 the remaining integral from below, we restrict the range of integration to the points $s=\gamma(y)$ in the disk $B(s_{0}, \delta/\log B)$, so that we may approximate $f$  by means of
  Lemma \ref{lem: approx f}. We have
\[
	f(\gamma(y))-f(s_{0}) = \frac{f''(s_{0})}{2}(\gamma(y)-s_{0})^{2}(1+\lambda_{0}(\gamma(y))).
\]
Now $f''(s_{0})$ is real and $f''(s_{0}) = \log B\log x\bigl(1+O((\log_{2}B)^{-1})\bigr)$ and 
\[
	(\gamma(y)-s_{0})^{2}(1+\lambda_{0}(\gamma(y))) = - \abs{\gamma(y)-s_{0}}^{2}\abs{1+\lambda_{0}(\gamma(y))} \ge -2y^{2},
\]
if we take a value for $\delta$ provided by Lemma \ref{lem: approx f} corresponding to the choice $\eps=1$ say. Hence
the integral $\int_{y^{-}}^{y^{+}}\exp\bigl( f(\gamma(y))-f(s_{0})\bigr)\dif y$
is bounded from below by 
\[
	\int_{-\delta/\log B}^{\delta/\log B}\exp\bigl(-2(\log B\log x) y^{2}\bigr)\dif y \gg_{\delta} \min\biggl(\frac{1}{\log B}, \frac{1}{\sqrt{\log B\log x}}\biggr) = \frac{1}{\sqrt{\log B\log x}}.
\]
We conclude that the contribution from $s_{0}$ has sign $(-1)^{K}$ and has absolute value bounded from below by
\begin{equation}
\label{eq: contribution s_{0}}
	\frac{x}{\tau}\exp\biggl(-(1-\sigma_{0})\log x + \int_{s_{0}}^{\infty}\eta(z)\dif z + O(\log_{2}x)\biggr).
\end{equation}
Let us now estimate $\int_{s}^{\infty}\eta(z)\dif z$.
 We use the representation \eqref{eq: formula int eta} and integrate by parts three times,
\begin{align}
	\int_{s}^{\infty}\eta(z)\dif z	&= \frac{B^{1+\I\tau-s}-2B^{(1+\I\tau-s)/2}}{4(1+\I\tau -s)\log B} + \frac{B^{1+\I\tau-s}-4B^{(1+\I\tau-s)/2}}{4((1+\I\tau -s)\log B)^{2}} \nonumber \\
						&+ \frac{B^{1+\I\tau-s}-8B^{(1+\I\tau-s)/2}}{2((1+\I\tau -s)\log B)^{3}} + \frac{3}{2((1+\I\tau -s )\log B)^{3}}\int_{1/2}^{1}\frac{B^{(1+\I\tau-s)u}}{u^{4}}\dif u 
						\label{eq: estimate int eta}.
\end{align}
Although we 
did not have to perform partial integration to obtain the contribution \eqref{eq: contribution s_{0} explicit} from $s_{0}$ below, we shall require these finer estimates for $\int^{\infty}_{s} \eta(z) \dif z$ later on.
For $s=s_{0}$ we get
\begin{align}
	\int_{s_{0}}^{\infty}\eta(z)\dif z	
		&= \frac{B^{1-\sigma_{0}}}{4(1-\sigma_{0})}\frac{1}{\log B} + \frac{B^{1-\sigma_{0}}}{4(1-\sigma_{0})}\frac{1}{(1-\sigma_{0})(\log B)^{2}}  \nonumber \\
		&+ \frac{B^{1-\sigma_{0}}}{4(1-\sigma_{0})}\frac{2}{(1-\sigma_{0})^{2}(\log B)^{3}} + O\biggl(\frac{B^{1-\sigma_{0}}}{1-\sigma_{0}}\frac{1}{(1-\sigma_{0})^{3}(\log B)^{4}}\biggr) \nonumber \\
		&= \frac{\log x}{\log B}\biggl(1+\frac{1}{(1-\sigma_{0})\log B}+\frac{2}{((1-\sigma_{0})\log B)^{2}} + O\biggl(\frac{1}{((1-\sigma_{0})\log B)^{3}}\biggr)\biggr), \label{eq: estimate int eta at s_{0}}
\end{align}
where we have used \eqref{eq: approx log x}. Combining the above with the estimate \eqref{eq: approx sigma} for $\sigma_{0}$ and the relations \eqref{eq: def tau} and \eqref{eq: def x} between $\tau$ and $B$, and $x$ and $B$ respectively, we get that the contribution from $s_{0}$ has absolute value which is bounded from below by
\begin{equation}
\label{eq: contribution s_{0} explicit}
	x\exp\biggl\{-(c(\alpha+1))^{\frac{1}{\alpha+1}}(\log x\log_{2}x)^{\frac{\alpha}{\alpha+1}}\biggl(1+\frac{\alpha}{\alpha+1}\frac{\log_{3}x}{\log_{2}x}+ O\biggl(\frac{1}{\log_{2}x}\biggr)\biggr)\biggr\}.
\end{equation}

\subsection{The steepest paths through $s_{m}$, $m\neq0$.}
We now consider the contributions from the other saddle points. In this case by such contributions we mean
\[
	\frac{1}{\pi}\Im\int_{\Gamma_{m}}\e^{f(s)}g(s)\dif s,
\]
where $\Gamma_{m}$ is some contour which connects the two horizontal lines $t=t_{m}^{-}$ and $t=t_{m}^{+}$. This contribution will be of lower order than 
that of $s_{0}$. We shall again use the method of steepest descent; just taking some simple choice for $\Gamma_{m}$ (e.g.\ a vertical line segment) and estimating the integral via the triangle inequality appears to be insufficient for small $m$. We consider $\abs{m}\le M \coloneqq \lfloor(\log_{2}B)^{3/4}\rfloor$. The part of the Perron integral where $t<t_{-M}^{-}$ or $t>t_{M}^{+}$ can be estimated without appealing to the saddle point method, and this will be done in the next section.

We want to show that we can connect the two lines $t=t_{m}^{-}$ and $t=t_{m}^{+}$ with the path of steepest decent through $s_{m}$. We first consider the steepest path in a small neighborhood of $s_{m}$. By applying Lemma \ref{lem: approx f} with $\eps=1/5$, we find some $\delta' > 0$ (independent of $K$ and $m$) such that 
\[
	f(s) - f(s_{m}) = \frac{f''(s_{m})}{2}(s-s_{m})^{2}(1+\lambda_{m}(s)) \eqqcolon (\psi_{m}(s))^{2}, 
\]
where $\abs{\lambda_{m}(s)} < 1/5$ for $s\in B(s_{m}, \delta'/\log B)$, and where $\psi_{m}$ is a holomorphic bijection of $B(s_{m}, \delta'/\log B)$ onto some neighborhood $U$ of $0$. The path of steepest descent $\Gamma_{m}$ in $B(s_{m}, \delta'/\log B)$ is the inverse image under $\psi_{m}$ of the curve $\{ z\in U: \Re z=0\}$. Since $f''(s_{m}) = \log B\log x\bigl(1+O((\log_{2}B)^{-1})\bigr)$ (which follows from \eqref{eq: approx log x}), we have that 
\[ 	
	\Re\bigl( f(s) - f(s_{m}) \bigr) = \frac{\abs{f''(s_{m})}}{2}r^{2}\bigl((1+\Re\lambda_{m}(s))\cos2\phi - (\Im\lambda_{m}(s))\sin2\phi + O((\log_{2}B)^{-1})\bigr),
\]
where we have set $s-s_{m}=r\e^{\I\phi}$. Points $s \in \Gamma_{m}\setminus\{s_{m}\}$ satisfy $\Re\bigl( f(s) - f(s_{m}) \bigr) < 0$, and since $\abs{\lambda_{m}(s)} < 1/5$, it follows from the above equation that such points lie in the union of the sectors $\phi \in (\pi/5, 4\pi/5) \cup (-\pi/5, -4\pi/5)$, say. We have that $\Gamma_{m}\setminus\{s_{m}\}$ is the union of two curves $\Gamma_{m}^{+}$ and $\Gamma_{m}^{-}$ where $\Gamma_{m}^{+}$ lies in the sector $\phi \in (\pi/5, 4\pi/5)$, and $\Gamma_{m}^{-}$ lies in the sector $\phi \in (-\pi/5, -4\pi/5)$. (It is impossible that both pieces lie in the same sector, since the angle between $\Gamma_{m}^{+}$ and $\Gamma_{m}^{-}$ at $s_{m}$ equals $\pi$, as $\psi_{m}^{-1}$ is 
conformal.)
Both $\Gamma_{m}^{+}$ and $\Gamma_{m}^{-}$ intersect the circle $\partial B(s_{m}, \delta'/(2\log B))$, which can be seen from the fact that $\psi_{m}(\Gamma_{m}^{+})$ and $\psi_{m}(\Gamma_{m}^{-})$ both intersect the closed curve $\psi_{m}(\partial B(s_{m}, \delta'/(2\log B)))$. 
From this it follows that the path of steepest descent $\Gamma_{m}$ connects the lines $t=t_{m}-\delta/\log B$ and $t=t_{m}+\delta/\log B$, where $\delta=(\delta'/2)\sin(\pi/5)$. 
Since $f'(s) = f''(s_{m})(s-s_{m})(1+\tilde{\lambda}_{m}(s))$, with also $\abs[1]{\tilde{\lambda}_{m}(s)} < 1/5$, it follows that $\arg f'(s) \in (\pi/10, 9\pi/10)$ if $\phi\in(\pi/5, 4\pi/5)$, and $\arg f'(s) \in (-9\pi/10, -\pi/10)$ if $\phi\in(-4\pi/5, -\pi/5)$. This implies that the tangent vector of $\Gamma_{m}$ has argument contained in $(\pi/10, 9\pi/10)$ (when $\Gamma_{m}$ is parametrized in such a way that we move in the upward direction). From this it follows that the length of $\Gamma_{m}$ in the neighborhood $B(s_{m}, \delta'/(2\log B))$ is bounded by $O(\delta/\log B)$.

For the continuation of $\Gamma_{m}$ outside this neighborhood of $s_{m}$, we argue as follows. We again set $\theta=(t-\tau)\log B$, and we consider the range 
\begin{equation}
\label{eq: range theta}
	\theta \in [2\pi m - \pi/2, 2\pi m + \pi/2] \setminus [2\pi m -\delta/2, 2\pi m + \delta/2].
\end{equation}
The equation for the steepest paths through $s_{m}$, $\Im f(s) = \Im f(s_{m})$, gives 
\[
	t_{m}\log x - \frac{1}{4}\int_{1/2}^{1}B^{(1-\sigma_{m})u}\sin\bigl((t_{m}-\tau)(\log B)u\bigr)\frac{\dif u}{u} = t\log x - \frac{1}{4}\int_{1/2}^{1}B^{(1-\sigma)u}\sin(\theta u)\frac{\dif u}{u}, 
\]
which is equivalent to
\begin{equation}
\label{eq: steepest path s_{m}}
	(t-t_{m})\log x + \frac{1}{4}\int_{1/2}^{1}B^{(1-\sigma_{m})u}\sin\bigl((t_{m}-\tau)(\log B)u\bigr)\frac{\dif u}{u} = \frac{1}{4}\int_{1/2}^{1}B^{(1-\sigma)u}\sin(\theta u)\frac{\dif u}{u}. 
\end{equation}
Also the points on the path of steepest \emph{ascent} satisfy this equation, but we will show that for fixed $\theta$ in the range \eqref{eq: range theta}, the above equation has a unique solution for $\sigma$ (in a sufficiently large range for $\sigma$ that contains $\sigma_{m}$). These solutions necessarily form the continuation of the path of steepest descent in the neighborhood $B(s_{m}, \delta'/(2\log B))$.

We consider $\theta$ in the range \eqref{eq: range theta} fixed (so also $t$ is fixed). We have $\sin\theta \gg_{\delta} 1$. The right hand side of \eqref{eq: steepest path s_{m}} is a monotone function of $\sigma$ for $\sigma$ in the range $\sigma= 1-\alpha\bigl(\log_{2}B + O(1)\bigr)/\log B$:
\begin{align*}
	\dpd{\RHS}{\sigma}	&= -\frac{1}{4}\int_{1/2}^{1}B^{(1-\sigma)u}\log B\sin(\theta u)\dif u \\
					&=-\frac{1}{4}\frac{B^{1-\sigma}}{(1-\sigma)^{2}+(\theta/\log B)^{2}}\biggl((1-\sigma)\sin\theta - \frac{\theta\cos\theta}{\log B}\biggr)\bigl(1+O_{\delta}(B^{(\sigma-1)/2})\bigr).
\end{align*}
Since $\abs{\theta} \ll (\log_{2}B)^{3/4}$, this indeed has a fixed sign in the aforementioned range. By setting $\sigma=\sigma_{m}-a/\log B$ for some large positive and negative values of $a$, one can conclude that \eqref{eq: steepest path s_{m}} has a unique solution. Indeed, integrating by parts gives
\begin{align*}
	\LHS		&= (t-t_{m})\log x + O\biggl(\frac{\log x}{\log B}\frac{\abs{m}}{\log_{2}B}\biggr), \\
	\RHS	&= \e^{a}\frac{\log x}{\log B}\sin\theta\biggl(1+O_{\delta}\biggl(\frac{\abs{m}}{\log_{2}B}\biggr)\biggr).
\end{align*}
Here we used that 
\[
	\sin((t_{m}-\tau)\log B) \ll \frac{\abs{m}}{\log_{2}B}, \quad \frac{B^{1-\sigma}}{4(1-\sigma)} = \e^{a}\log x\biggl(1+O\biggl(\frac{\abs{m}}{\log_{2}B}\biggr)\biggr),
\]
by \eqref{eq: approx t} and \eqref{eq: approx log x}, \eqref{eq: approx sigma}, \eqref{eq: approx t} respectively.
Since $t-t_{m} = (\theta-2\pi m)/\log B + O(\,\abs{m}/(\log B\log_{2}B))$ by \eqref{eq: approx t}, it follows that $\LHS \lessgtr \RHS$ if $a$ is sufficiently large, resp.\ small. This shows that we can connect the lines $t=t_{m}^{-}$ and $t=t_{m}^{+}$ with the path of steepest descent $\Gamma_{m}$.

Denoting the solutions of \eqref{eq: steepest path s_{m}} for $\sigma$ at $\theta = 2\pi m \pm \pi/2$ by $\sigma_{m}^{\pm}$, and setting $\sigma_{m}^{\pm} = \sigma_{m} - a_{m}^{\pm}/\log B$, the above calculations also show that 
\begin{equation}
\label{eq: approx sigma_{m}^{pm}}
	a_{m}^{\pm} = \log\frac{\pi}{2} + O\biggl(\frac{\abs{m}}{\log_{2}B}\biggr), \quad \mbox{so} \quad \sigma_{m}^{\pm} = \sigma_{m} - \frac{\log(\pi/2)}{\log B} + O\bigl((\log B)^{-1}(\log_{2}B)^{-1/4}\bigr). 
\end{equation}

Finally we need that the length of $\Gamma_{m}$ is not too large. For the part inside the neighborhood $B(s_{m}, \delta'/(2\log B))$, this was already remarked at the beginning of this subsection. Outside
 this neighborhood, we use that $\pd{}{\sigma}\RHS \gg_{\delta} \log x$, $\pd{}{\theta}\RHS \ll \log x/\log B$ and $\pd{}{\theta}\LHS = \log x/\log B$, so that $\od{}{\theta}\sigma(\theta) \ll_{\delta} 1/\log B$. This implies that $\length(\Gamma_{m}) \ll 1/\log B$.

\subsection{The contributions from $s_{m}$, $m\neq0$}
\label{subsection: contributions from s_m}
On the path of steepest descent $\Gamma_{m}$, $\Re f$ reaches its maximum at $s_{m}$. This together with Lemma \ref{lem: bound g} implies the following bound for the contribution of $s_{m}$, $m\neq0$:
\[
	\Im \frac{1}{\pi}\int_{\Gamma_{m}}\e^{f(s)}g(s)\dif s \ll \frac{x}{\tau}\exp\biggl(-(1-\sigma_{m})\log x + \Re \int_{s_{m}}^{\infty}\eta(z)\dif z\biggr) \length(\Gamma_{m}).
\]
Using \eqref{eq: estimate int eta}, \eqref{eq: approx log x}, the inequality $\abs{1+\I\tau-s_{m}} > 1-\sigma_{0}$, and \eqref{eq: estimate int eta at s_{0}}, we get 
\begin{align*}
	\Re\int_{s_{m}}^{\infty}\eta(z)\dif z	&\le \frac{\log x}{\log B}\biggl(1+\frac{1}{\abs{1+\I\tau -s_{m}}\log B} 
										+ \frac{2}{(\,\abs{1+\I\tau -s_{m}}\log B)^{2}} + O\biggl(\frac{1}{(\,\abs{1+\I\tau -s_{m}}\log B)^{3}}\biggr)\biggr) \\
								&\le \int_{s_{0}}^{\infty}\eta(z)\dif z + O\biggl(\frac{\log x}{(\log B)(\log_{2}B)^{3}}\biggr).
\end{align*}
Combining this with Lemma \ref{lem: finer estimate sigma}, we see that the contribution of $s_{m}$ is bounded by 
\[
	\frac{x}{\tau}\exp\biggl(-(1-\sigma_{0})\log x + \int_{s_{0}}^{\infty}\eta(z)\dif z - d\frac{\log x}{\log B(\log_{2}B)^{2}} + O\biggl(\frac{\log x}{\log B(\log_{2}B)^{3}}\biggr)\biggr).
\]
Since 
\[
	\frac{\log x}{\log B(\log_{2}B)^{2}} \asymp \frac{(\log x)^{\frac{\alpha}{\alpha+1}}}{(\log_{2}x)^{\frac{2\alpha+3}{\alpha+1}}}
\]
tends to infinity, this is of strictly lower order than the contribution of $s_{0}$, \eqref{eq: contribution s_{0}}. The same holds for $\sum_{0<\,\abs{m}\le M}\int_{\Gamma_{m}}\e^{f(z)}g(z)\dif z$, since summing all these contributions enlarges the bound only by a factor 
$M=\exp(O(\log_{3}x))$.

\section{The remainder in the contour integral}
\label{sec: The remainder}
Let us recall that the main goal is to estimate the Perron integral
\[
	\frac{1}{2\pi\I}\int\zeta_{C,K}(s)\frac{x_{K}^{s}}{s}\dif s = \frac{1}{2\pi\I}\int\e^{f(s)}g(s)\dif s,
\]
where the integral is along some suitable contour connecting the points $\kappa\pm\I T$ for some $\kappa>1$, $T>0$, which will be specified later. We refer again to the definitions of $f$ and $g$: \eqref{eq: def f} and \eqref{eq: def g}. In the previous section, we have used the fact that $\zeta_{C,K}$ is very large near the saddle point $s_{0}$ to show that the integral along a small contour $\Gamma_{0}$ passing through $s_{0}$ is also very large. This should be considered the ``main term'' in our estimate for the Perron integral. The zeta function is also large around the other saddle points $s_{m}$, $m\neq0$, but since these are slightly to the left of $s_{0}$, $x^{s}$ is smaller there. This turned out to be enough to show that the integrals along similar contours $\Gamma_{m}$ through $s_{m}$, $m\neq0$ combined are of lower order than the main term.
	
In this section, we estimate ``the remainder'', which consists of three parts. First we have to connect the steepest paths $\Gamma_{m}$ to each other. This forms one contour near the saddle points, which we have to connect to the ``standard'' Perron contour $[\kappa-\I T, \kappa+\I T]$. Finally, we also have to estimate the remainder in the effective Perron formula \eqref{eq: Perron inversion}.

\subsection{Connecting the steepest paths}
Let $\Upsilon_{m}$ be the line segment connecting $\sigma_{m-1}^{+}+\I t_{m-1}^{+}$ to $\sigma_{m}^{-}+\I t_{m}^{-}$ if $m>0$, and connecting $\sigma_{m}^{+}+\I t_{m}^{+}$ to $\sigma_{m+1}^{-}+\I t_{m+1}^{-}$ if $m<0$. By previous calculations (\eqref{eq: a_{theta}} and \eqref{eq: approx sigma_{m}^{pm}}), we know that the real part on these lines is bounded by $\sigma_{0} - \frac{\log(\pi/2)}{2\log B}$\textcolor{blue}{,} say. Furthermore, $\Re \int_{s}^{\infty}\eta(z)\dif z$ is significantly smaller on these lines than at the saddle points. Indeed, using \eqref{eq: estimate int eta} and the fact that 
\[
	\Re \frac{B^{1+\I\tau-s}}{1+\I\tau-s} = \frac{B^{1-\sigma}}{(1-\sigma)^{2}+(t-\tau)^{2}}\biggl(\cos\bigl((t-\tau)\log B\bigr)(1-\sigma) + (t-\tau)\sin\bigl((t-\tau)\log B\bigr)\biggr),
\]
we have 
\begin{align}
	\Re\int_{s}^{\infty}\eta(z)\dif z	&= \Re \frac{B^{1+\I\tau-s} - 2B^{(1+\I\tau-s)/2}}{4(1+\I\tau-s)\log B} + O\biggl(\frac{B^{1-\sigma}}{(\log_{2}B)^{2}}\biggr) \nonumber \\
							&\le \frac{(t-\tau)B^{1-\sigma}}{4(1-\sigma)^{2}\log B} + O\biggl(\frac{B^{1-\sigma}}{(\log_{2}B)^{2}}\biggr) \nonumber \\
							&\ll \frac{\log x}{(\log B)(\log_{2}B)^{1/4}} \label{eq: bound on Upsilon},
\end{align}
for $s\in\Upsilon_{m}$. In the first inequality we used that $\cos\bigl((t-\tau)\log B\bigr) \le 0$, and for the second estimate we used \eqref{eq: approx log x} and that $\sigma-\sigma_{0}\ll 1/\log B$ (which follows from \eqref{eq: approx sigma_{m}^{pm}} and \eqref{eq: approx sigma}), together with $(t-\tau)/(1-\sigma) \ll (\log_{2}B)^{-1/4}$. Using Lemma \ref{lem: bound g} to bound $g$, we see that 
\[
	\sum_{0 < \, \abs{m} \leq M} \int_{\Upsilon_{m}  }\e^{f(s)}g(s)\dif s \ll \frac{x}{\tau}\exp\biggl(-(1-\sigma_{0})\log x -\frac{\log(\pi/2)}{2}\frac{\log x}{\log B} + O\biggl(\frac{\log x}{(\log B)(\log_{2}B)^{1/4}}\biggr)\biggr), 
\]
which is negligible with respect to the contribution from $s_{0}$, in view of \eqref{eq: contribution s_{0}} and \eqref{eq: estimate int eta at s_{0}}.

\subsection{Returning to the line $[\kappa-\I T, \kappa+\I T]$}
We will now connect the contour near the saddle points to the line $[\kappa-\I T, \kappa+\I T]$. First we need another lemma to bound $g$.
\begin{lemma}
\label{lem: bound g (2)}
Suppose $s=\sigma+\I t$ satisfies 
\[
	\sigma \ge 1 - O\left(\frac{\log_{2}B_{K}}{\log B_{K}}\right), \quad t\ge0. 
\]
Then,
\[
	\sum_{k=0}^{K-1}\int_{s}^{s+1}\bigl(\eta_{k}(z) + \tilde{\eta}_{k}(z) + \xi_{k}(z)\bigr)\dif z + \int_{s}^{s+1}\bigl(\tilde{\eta}_{K}(z)+\xi_{K}(z)\bigr)\dif z - \int_{s+1}^{\infty}\eta_{K}(z)\dif z \ll 1.
\]
\end{lemma}
\begin{proof}
The sum of the integrals $\int_{s+1}^{\infty}$ is trivially bounded. Recall that 
\[
	\int_{s}^{\infty}\eta_{k}(z)\dif z = \frac{1}{4}\int_{s}^{\infty}\frac{B_{k}^{1-z}-B_{k}^{(1-z)/2}}{1+\I\tau_{k}-z}\dif z = \frac{1}{4}\int_{1/2}^{1}\frac{B_{k}^{(1+\I\tau_{k}-s)u}}{u}\dif u.
\]
Let $k<K$.

\noindent\textbf{Case 1: $t\le\tau_{k}/2$ or $t\ge2\tau_{k}$.} Then the above integral is bounded by
\[
	\frac{B_{k}^{1-\sigma}}{\tau_{k}\log B_{k}} \le \frac{1}{\tau_{k}}\exp\biggl\{O\biggl(\frac{\log_{2}B_{K}}{\log B_{K}}\log B_{k}\biggr)\biggr\} \ll \frac{1}{\tau_{k}},
\]
where the fast growth of $(B_{k})_{k}$ was used (property \ref{Property (a)}).

\noindent\textbf{Case 2: $\tau_{k}/2 < t < 2\tau_{k}$.} Then we use the second integral representation for $\int^{\infty}_{s} \eta_{k}(z) \dif z$ and get the bound $B_{k}^{1-\sigma} \ll 1$. This case occurs at most once.

Since $\sum_{k}(1/\tau_{k})$ converges, this deals with the terms involving $\eta_{k}$; bounding the terms with $\tilde{\eta}_{k}$, $k<K$ is completely analogous, except that in this case we can always use the bound from \textbf{Case 1} since $\abs{1-\I\tau_{k}-s} \gg \tau_{k}$ (since $t\ge0$). Also 
\[
	\int_{s}^{\infty}\tilde{\eta}_{K}(z)\dif z \ll \frac{1}{\tau_{K}}\exp(O(\log_{2}B_{K})) = \exp\bigl(O(\log_{2}B_{K}) - c(\log B_{K})^{\alpha}\bigr) \ll 1.
\]

Finally for $k\le K$, 
\begin{align*}
	\int_{s}^{\infty}\xi_{k}(z)\dif z	&= -\frac{1}{2}\int_{1}^{\log C_{k}/\log B_{k}}\frac{B_{k}^{(1-s)u}}{u}\dif u \ll \biggl(\frac{\log C_{k}}{\log B_{k}}-1\biggr)C_{k}^{1-\sigma} \\
							&\ll \exp\biggl\{-2c(\log B_{k})^{\alpha} + O\biggl(\frac{\log_{2}B_{K}}{\log B_{K}}\log C_{k}\biggr)\biggr\} \ll \exp\bigl(-c(\log B_{k})^{\alpha}\bigr) = \frac{1}{\tau_{k}}, 
\end{align*}
where we used \eqref{eq: approx C}.
\end{proof}

Recall that we have set $M = \lfloor(\log_{2}B)^{3/4}\rfloor$. Set $T_{1}^{\pm} = t_{\pm M}^{\pm}$. We now connect the point $\sigma_{-M}^{-}+ \I T_{1}^{-}$ to some point on the real axis\footnote{The ``complete'' contour will consist of the contour described in this section in the upper half plane, together with its reflection across the real axis in the lower half plane. As mentioned before, it suffices to only consider the part in the upper half plane, since $\zeta_{C,K}(\overline{s})=\overline{\zeta_{C,K}(s)}$.}, and $\sigma_{M}^{+}+\I T_{1}^{+}$ to the point $\kappa+\I T$ by a number of line segments ($\kappa$ and $T$ will be specified later). In what follows,
 we will use expressions in the style ``The segment $\Delta$ contributes $\ll F$, which is negligible'', by which we mean that $\int_{\Delta}\e^{f(s)}g(s)\dif s \ll F$ and that $F$ is of lower order than the contribution of $s_{0}$ \eqref{eq: contribution s_{0}}. We will also apply Lemma \ref{lem: bound g (2)} repeatedly, without referring to it each time.

\mbox{}

First we connect $\sigma_{M}^{+}+\I T_{1}^{+}$ to $\sigma_{0}+\I T_{1}^{+}$, and similarly $\sigma_{-M}^{-}+\I T_{1}^{-}$ to $\sigma_{0}+\I T_{1}^{-}$. By \eqref{eq: bound on Upsilon}, this contributes 
\[
	\ll \frac{x^{\sigma_{0}}}{\tau}\exp\biggl(O\biggl(\frac{\log x}{(\log B)(\log_{2}B)^{1/4}}\biggr)\biggr),
\]
which is negligible. Next, set $T_{2}^{\pm} = \tau \pm \exp\bigl((\log B)^{\alpha/2}\bigr)$, $\Delta_{1}^{+}=[\sigma_{0}+\I T_{1}^{+}, \sigma_{0}+\I T_{2}^{+}]$, $\Delta_{1}^{-}=[\sigma_{0}+\I T_{2}^{-}, \sigma_{0}+ \I T_{1}^{-}]$. We require a better bound for $\int_{s}^{\infty}\eta(z)\dif z$ on these lines. Integrating by parts, one sees that 
\[
	\int_{s}^{\infty}\eta(z)\dif z = \frac{1}{4}\int_{s}^{\infty}\frac{B^{1-z}-B^{(1-z)/2}}{1+\I\tau-z}\dif z = \frac{B^{1-s} - 2B^{(1-s)/2}}{4(1+\I\tau-s)(\log B)} + 
			O\biggl(\frac{(\log B)^{\alpha}}{(\log_{2}B)^{2}}\biggr),
\]
if $\Re s=\sigma_{0}$. If $\abs{t-\tau_{K}} \ge (\log_{2}B)^{3/4}/(2\log B)$ say, then for some $r>0$,
\[
	\frac{1}{\abs{1+\I\tau-s}} \le \frac{1}{1-\sigma_{0}}\biggl(1-r\biggl(\frac{t-\tau}{1-\sigma_{0}}\biggr)^{2}\biggr) \le \frac{1}{1-\sigma_{0}}\biggl(1-\frac{r/4}{(\log_{2}B)^{1/2}}\biggr).
\]
Hence,
\[
	\Re \int_{s}^{\infty}\eta(z)\dif z \le \frac{\log x}{\log B}\biggl(1-\frac{r/4}{(\log_{2}B)^{1/2}}\biggr) + O\biggl(\frac{\log x}{(\log B)(\log_{2}B)}\biggr).
\]
If furthermore $\abs{t-\tau}\ge1$, then 
\[
	\Re \int_{s}^{\infty}\eta(z)\dif z \ll \frac{B^{1-\sigma_{0}}}{\log B} \asymp (\log B)^{\alpha-1} \ll 1.
\]
These bounds imply that the contribution from $\Delta_{1}^{\pm}$ is 
\[
	\ll \frac{x^{\sigma_{0}}}{\tau}\biggl\{\exp\biggl(\frac{\log x}{\log B}\biggl(1-\frac{r/4}{(\log_{2}B)^{1/2}}\biggr) + O\biggl(\frac{\log x}{(\log B)(\log_{2}B)}\biggr)\biggr)
									+ \exp\bigl((\log B)^{\alpha/2}\bigr) \biggr\},
\]
which is admissible.
Next, we set 
\[
	\sigma' = \sigma_{0} -2\frac{c(\log B)^{\alpha}}{\log x} = \sigma_{0} - O\biggl(\frac{\log_{2}B}{\log B}\biggr),
\]
so that $x^{\sigma'} = x^{\sigma_{0}}/\tau^{2}$. Set $\Delta_{2}^{\pm} = [\sigma'+\I T_{2}^{\pm}, \sigma_{0}+\I T_{2}^{\pm}]$. For $\sigma \ge 1 - O\bigl(\log_{2}B/\log B\bigr)$ and $\abs{t-\tau} \ge \exp\bigl((\log B)^{\alpha/2}\bigr)$, 
\[
	\Re \int_{s}^{\infty}\eta(z)\dif z \ll \exp\bigl(-(\log B)^{\alpha/2} + O(\log_{2}B)\bigr) \ll 1, 
\]
so the contribution from $\Delta_{2}^{\pm}$ is $\ll x^{\sigma_{0}}/\tau$, which is negligible. Let now $T_{3}^{+} = x^{2}$, $\Delta_{3}^{+} = [\sigma' + \I T_{2}^{+}, \sigma'+\I T_{3}^{+}]$, and $\Delta_{3}^{-} = [\sigma', \sigma'+ \I T_{2}^{-}]$. We have that 
\begin{align*}
	\int_{\Delta_{3}^{+}} 	&\ll x^{\sigma'}\int_{T_{2}^{+}}^{T_{3}^{+}}\frac{\dif t}{t} \ll \frac{x^{\sigma_{0}}}{\tau^{2}}\log x, \\
	\int_{\Delta_{3}^{-}}	&\ll x^{\sigma'}\biggl(\int_{1}^{T_{2}^{-}}\frac{\dif t}{t} + \frac{1}{\abs{\sigma'-1}}\biggr) \ll \frac{x^{\sigma_{0}}}{\tau^{2}}\biggl( (\log B)^{\alpha} + \frac{\log B}{\log_{2}B}\biggr).
\end{align*}
Both of these are admissible.
Finally we set $\Delta_{4}^{+} = [\sigma'+\I T_{3}^{+}, 3/2 + \I T_{3}^{+}]$. This segment only contributes $\ll x^{3/2}/T_{3}^{+} = 1/\sqrt{x}$.

We have now connected our contour to the line $[\kappa-\I T, \kappa+\I T]$, with $\kappa=3/2$ and $T=T_{3}^{+}=x^{2}$.

\section{Conclusion of the analysis of the continuous example} 
\label{Section conclusion continuous example}
By an effective Perron formula, e.g.\ \cite[Theorem II.2.3]{Tenenbaumbook}, we have that\footnote{The theorem in \cite{Tenenbaumbook} is only formulated in terms of discrete measures $\dif A = \sum_{n}a_{n}\delta_{n}$. One can easily verify that the result holds for general measures of locally bounded variation $\dif A$, upon replacing $\sum_{n}\dotso\abs{a_{n}}$ by $\int_{1^{-}}^{\infty} \dotso \abs{\dif A}$.} 
\begin{align*}
	N_{C,K}(x)		&= \frac{1}{2}\bigl(N_{C,K}(x^{+})+N_{C,K}(x^{-})\bigr) \\
				&= \frac{1}{2\pi\I}\int_{\kappa-\I T}^{\kappa+\I T}\zeta_{C,K}(s)\frac{x^{s}}{s}\dif s + 
					O\biggl(x^{\kappa}\int_{1^{-}}^{\infty}\frac{1}{u^{\kappa}\bigl(1+T\abs{\log(x/u)}\bigr)}\dif N_{C,K}(u)\biggr).
\end{align*}
We apply it with $x=x_{K}$, $\kappa=3/2$, and $T=(x_{K})^{2}$. 
Let us first deal with the error term in the effective Perron formula. We have for every $K$:
\[
	\dif N_{C,K}(u) = \expm(\dif\Pi_{C,K}(u)) \le \expm(2\dif\Li(u)) = (\delta_{1}(u)+\dif u)\ast(\delta_{1}(u)+\dif u) = \delta_{1}(u) + 2\dif u + \log u\dif u.
\]
Hence this error term is bounded by 
\begin{align*}
	&\phantom{\ll}\frac{x^{3/2}}{T\log x} + x^{3/2}\biggl(\int_{1}^{x/2} + \int_{x/2}^{x-1} + \int_{x-1}^{x+1} + \int_{x+1}^{\infty}\biggr)\frac{2 + \log u}{u^{3/2}\bigl(1+T\abs{\log(x/u)}\bigr)}\dif u \\
	&\ll \frac{1}{\sqrt{x}\log x} + x^{3/2}\biggl(\frac{1}{x^{2}} + \frac{\log x}{x^{3/2}}\biggr) \ll \log x.
\end{align*}

We shift the contour in the integral to the contour described in the previous (sub)sections. We showed that the integral along the shifted contour has sign $(-1)^{K}$, and has absolute value bounded from below by 
\[
	x_{K}\exp\biggl\{-(c(\alpha+1))^{\frac{1}{\alpha+1}}(\log x_{K}\log_{2}x_{K})^{\frac{\alpha}{\alpha+1}}\biggl(1+\frac{\alpha}{\alpha+1}\frac{\log_{3}x_{K}}{\log_{2}x_{K}}+ O\biggl(\frac{1}{\log_{2}x_{K}}\biggr)\biggr)\biggr\},
\]
see \eqref{eq: contribution s_{0} explicit}. Shifting the contour also gives a contribution from the pole at $s=1$, which is $\rho_{C,K}x_{K}$, where 
\[
	\rho_{C,K} = \res_{s=1}\zeta_{C,K}(s) = \exp\biggl(\sum_{k=0}^{K}\int_{1}^{2}\bigl(\eta_{k}(z)+\tilde{\eta}_{k}(z)+\xi_{k}(z)\bigr)\dif z\biggr).
\]
To conclude the analysis of the continuous example $(\Pi_{C}, N_{C})$, we need to show that the oscillation result holds for $N_{C}$, i.e.\ that $N_{C}(x) - \rho_{C}x$ displays the desired oscillation. The density $\rho_{C}$ of $N_{C}$ equals the right hand residue of $\zeta_{C}$ at $s=1$, that is $\lim_{s\to 1^{+}}(s-1)\zeta_{C}(s)$ (see e.g. \cite[Theorem 7.3]{DiamondZhangbook}):
\[
	\rho_{C} =  \exp\biggl(\sum_{k=0}^{\infty}\int_{1}^{2}\bigl(\eta_{k}(z)+\tilde{\eta}_{k}(z)+\xi_{k}(z)\bigr)\dif z\biggr).
\]
Now 
\begin{gather*}
	\int_{1}^{2}\bigl(\eta_{k}(s) + \tilde{\eta}_{k}(z)\bigr)\dif z \ll \int_{1}^{2}\frac{B_{k}^{1-z}-B_{k}^{(1-z)/2}}{1 \pm\I\tau_{k}-z}\dif z \ll \frac{1}{\tau_{k}\log B_{k}}, \\
	 \int_{1}^{2}\xi_{k}(z)\dif z = \frac{1}{2}\int_{\log B_{k}}^{\log C_{k}}\frac{\e^{-u}-1}{u}\dif u \ll \frac{\log C_{k}-\log B_{k}}{\log B_{k}} \ll \frac{1}{\tau_{k}^{2}},
\end{gather*}
where we used \eqref{eq: approx C} in the last step. By property \ref{Property (a)}, we may assume that 
\[
	\sum_{k=K+1}^{\infty}\frac{1}{\tau_{k}\log B_{k}} \le \frac{2}{\tau_{K+1}\log B_{K+1}} \le \frac{1}{x_{K}}.
\]
Hence we have
\[
	\rho_{C,K} - \rho_{C} = \rho_{C}\biggl\{\exp\biggl(-\sum_{k=K+1}^{\infty}\int_{1}^{2}\bigl(\eta_{k}(z)+\tilde{\eta}_{k}(z)+\xi_{k}(z)\bigr) \dif z\biggr)-1\biggr\} \ll \frac{1}{x_{K}},
\]
so that
\begin{align*}
	N_{C}(x_{K}) -\rho_{C}x_{K} 	&= N_{C,K}(x_{K}) -\rho_{C,K}x_{K} + (\rho_{C,K}-\rho_{C})x_{K} \\
							&= \Omega_{\pm}\Bigl(x_{K}\exp\bigl(-(c(\alpha+1))^{\frac{1}{\alpha+1}}(\log x_{K}\log_{2}x_{K})^{\frac{\alpha}{\alpha+1}}(1+\dotso)\bigr)\Bigr) +O(1).
\end{align*}
This concludes the proof of the existence of a continuous Beurling prime system satisfying \eqref{eq: asymptotics Pi} and \eqref{eq: oscillation N}.

\section{The discrete example}
\label{sec: Discretization}
We will now show the existence of a \emph{discrete} Beurling prime system $(\Pi,N)$
arising from a sequence of Beurling primes $1<p_{1}\le p_{2}\le \dotso$ and satisfying \eqref{eq: asymptotics Pi} and \eqref{eq: oscillation N}. This will be done by approximating the continuous system $(\Pi_{C}, N_{C})$ with a discrete one via a probabilistic procedure devised by
 the first and third named authors in \cite{B-V2021}. This random approximation method is an improvement of that
  of Diamond, Montgomery, and Vorhauer \cite[Section 7]{DiamondMontgomeryVorhauer} (see also Zhang \cite[Section 2]{Zhang2007}). 
We also use a trick introduced by the authors in \cite[Section 6]{B-D-V2020} in order to control the argument of the zeta function at some specific points; this is done by adding a well-chosen prime finitely many times to the system.

Given a non-decreasing right-continuous function $F$, which tends to $\infty$ and satisfies $F(1)=0$ and $F(x) \ll x/\log x$, the approximation procedure from \cite{B-V2021} guarantees the existence of a sequence of Beurling primes $\MP_{D} = (p_{j})_{j}$ with counting function $\pi_{D}$ satisfying
\begin{gather}
	\abs{\pi_{D}(x)-F(x)} \ll 1, \label{eq: bound pi_{D}}\\
	\forall y\ge1, \forall t\ge0: \abs[4]{\sum_{p_{j}\le y}p_{j}^{-\I t} - \int_{1}^{y}u^{-\I t}\dif F(u)} \ll \sqrt{y} + \sqrt{\frac{y\log(\,\abs{t}+1)}{\log(y+1)}}. \label{eq: bound exp sum}	
\end{gather}

We will apply this with\footnote{If $\alpha<1$ or $\alpha=1$ and $c \leq 1/2$, we can apply the method with $F=\Pi_{C}$, since $\Pi_{D}(x) - \pi_{D}(x) \ll \sqrt{x}\ll x\exp\bigl(-c(\log x)^{\alpha}\bigr)$, so that Lemma \ref{lem: pi_{C}} is not needed. In this case, the method of Diamond, Montgomery, and Vorhauer, which yields \eqref{eq: bound exp sum} and \eqref{eq: bound pi_{D}} with the bound $1$ replaced by $\sqrt{x}$, also suffices.} $F=\pi_{C}$, where $\pi_{C}$ is defined as
\[
	\pi_{C}(x) = \sum_{\nu=1}^{\infty}\frac{\mu(\nu)}{\nu} \Pi_{C}(x^{1/\nu}), \quad \mbox{ so that } \quad \Pi_{C}(x) = \sum_{\nu=1}^{\infty}\frac{\pi_{C}(x^{1/\nu})}{\nu}.
\]
Here, $\mu$ stands for
the classical M\"obius function.

\begin{lemma}
\label{lem: pi_{C}}
The function $\pi_{C}$ is non-decreasing, right-continuous, tends to $\infty$\textcolor{blue}{,} and satisfies $\pi_{C}(1)=0$ and $\pi_{C}(x) \ll x/\log x$.
\end{lemma}

\begin{proof}
We only need to show that $\pi_{C}$ is non-decreasing, the other assertions are obvious. Using the series expansion $\Li(x) = \sum_{n=1}^{\infty}\frac{(\log x)^{n}}{n!n}$, we have
\[
	\pi_{C}(x) = \li(x) + \sum_{k=0}^{\infty}\sum_{\nu=1}^{\infty}\bigl(r_{k,\nu}(x) + s_{k,\nu}(x)\bigr), 
\]
where
\begin{align*}
	\li(x) 			&= \sum_{\nu=1}^{\infty}\frac{\mu(\nu)}{\nu}\Li(x^{1/\nu}) = \sum_{n=1}^{\infty}\frac{(\log x)^{n}}{n!n\zeta(n+1)} \quad \mbox{($\zeta$ being the ordinary Riemann zeta function)}; \\
	r_{k,\nu}(x) 	&= \begin{dcases}
					\mathrlap{\frac{\mu(\nu)}{2\nu}\int_{A_{k}}^{x^{1/\nu}}\frac{1-u^{-1}}{\log u}\cos(\tau_{k}\log u)\dif u}
					\phantom{\frac{\mu(\nu)}{2\nu}\biggl(\int_{A_{k}}^{B_{k}}\frac{1-u^{-1}}{\log u}\cos(\tau_{k}\log u)\dif u + \bigl(\Li(B_{k})-\Li(x^{1/\nu})\bigr)\biggr)}
								&\mbox{for } A_{k}^{\nu} \le x < B_{k}^{\nu}, \\
					0			&\mbox{otherwise;}
			\end{dcases}\\ 
	s_{k, \nu}(x)	&= \begin{dcases}
					\frac{\mu(\nu)}{2\nu}\biggl(\int_{A_{k}}^{B_{k}}\frac{1-u^{-1}}{\log u}\cos(\tau_{k}\log u)\dif u + \bigl(\Li(B_{k})-\Li(x^{1/\nu})\bigr)\biggr)	&\mbox{for } B_{k}^{\nu} \le x < C_{k}^{\nu}, \\
					0																									&\mbox{otherwise.}
			\end{dcases}
\end{align*}
Note that the notation $\li(x)$ is not standard: here it does not refer to (a variant of) the logarithmic integral, but rather $\li(x)$ relates to $\Li(x)$ in the same way as $\pi(x)$ relates to $\Pi(x)$.

We have $\supp(r_{k,\nu}+s_{k,\nu}) = [A_{k}^{\nu}, C_{k}^{\nu}] \eqqcolon I_{k, \nu}$. The function $\pi_{C}$ is absolutely continuous, so it will follow that it is non-decreasing if we show that $\pi_{C}'$ is non-negative. If $x$ is contained in no $I_{k,\nu}$, then $\pi_{C}'(x) = \li'(x) >0$. Suppose now the contrary, and let $m$ be the largest integer 
such that $x\in I_{k, m}$ for some $k\ge0$. Note that $m\le \log x/\log A_{0}$. Since for each $\nu\le m$, there is at most one value of $k$ for which $x\in I_{k,\nu}$, we have
\begin{align*}
	\abs[4]{\biggl(\sum_{k=0}^{\infty}\sum_{\nu=1}^{\infty}\bigl(r_{k,\nu}(x) + s_{k,\nu}(x)\bigr)\biggr)'} 	&\le \frac{1}{2}\sum_{\substack{k,\nu \\ x\in I_{k,\nu}}}\frac{1-x^{-1/\nu}}{\nu\log x}x^{1/\nu-1} \\
		&\le \frac{1}{2\log x}\sum_{\nu=1}^{m}\frac{x^{1/\nu-1}}{\nu} \le \frac{1}{2\log x}\biggl(1 + \frac{\log_{2}x}{\sqrt{x}}\biggr).	
\end{align*}
On the other hand, 
\[
	\li'(x) \ge \frac{1}{\zeta(2)}\frac{1-x^{-1}}{\log x} \ge 0.6\frac{1-x^{-1}}{\log x},
\]
and together with $x\ge A_{0}$, this implies that $\pi_{C}'(x) > 0$ (we may assume that $A_{0}$ is sufficiently large).
\end{proof}

Applying the discretization procedure to $F=\pi_{C}$ shows the existence of a sequence of Beurling primes $\MP_{D} = (p_{j})_{j}$ with counting function $\pi_{D}$ satisfying \eqref{eq: bound pi_{D}} and \eqref{eq: bound exp sum}. Denote the Riemann prime counting function of $\MP_{D}$ by $\Pi_{D}$, and set 
\[
	\dif\Pi_{D,K}(u) = \sum_{p_{j}^{\nu}<A_{K+1}}\frac{1}{\nu}\delta_{p_{j}^{\nu}}(u) + \chi_{[A_{K+1},\infty)}(u)\dif\Li(u),
\]
where $\chi_{E}$ denotes the characteristic function of the set $E$. Let $\log\zeta_{D,K}(s)$ be the Mellin-Stieltjes transform of $\dif\Pi_{D,K}$. 
Set
\[
	S_{l} = \biggl[l\frac{\pi}{80}-\frac{\pi}{160}, l\frac{\pi}{80}+\frac{\pi}{160}\biggr) + 2\pi\Z \quad \mbox{for } l=0,1,\dotso,159.
\]
Then for some $l$ (resp.\ $r$), we have that for infinitely many even (resp.\ odd) values of $K$
\[
	\Im\bigl(\log\zeta_{D,K}(1+\I\tau_{K}) - \log\zeta_{C,K}(1+\I\tau_{K})\bigr) \in S_{l} \quad (\mbox{resp. } S_{r}).
\]
Assume without loss of generality that $l\ge r$. Then there exists a number $q$, close to $80/\pi$, such that
\begin{align}
	\abs{\Im\bigl(-l\log(1-q^{-(1+\I\tau_{K})})\bigr) + \mathrlap{l}\phantom{r}\frac{\pi}{80}} 	
													&< \frac{\pi}{40} \quad \mbox{if $K$ is even}, \label{eq: even}\\
	\abs{\Im\bigl(-l\log(1-q^{-(1+\I\tau_{K})})\bigr) + r\frac{\pi}{80}} 	&< \frac{\pi}{40} \quad \mbox{if $K$ is odd}.\label{eq: odd}
\end{align}
We refer to \cite[Section 6]{B-D-V2020} for a proof of this statement. That
proof only requires some fast growth of the sequence $(\tau_{k})_{k}$, which we may assume.

We define our final prime system $\MP$ as the prime system obtained by adding the prime $q$ with multiplicity $l$ to the system $\MP_{D}$. Denote its Riemann prime counting function by $\Pi$, and its integer counting function by $N$. We have 
\[
	\Pi(x) = \Pi_{D}(x) + O(\log_{2}x) = \Pi_{C}(x) + O(\log_{2}x),	
\]	
where in the last step we used \eqref{eq: bound pi_{D}}. Since $\Pi_{C}$ satisfies \eqref{eq: asymptotics Pi}, it is clear that $\Pi$ also satisfies\footnote{Recall that in the case $\alpha=c=1$, we have altered the error term in the PNT \eqref{eq: asymptotics Pi} to
 $O(\log_{2}x)$.} \eqref{eq: asymptotics Pi}. 

Set\footnote{This is a slight abuse of notation, since the equality $\Pi_{K}(u) = \sum_{\nu}\pi_{K}(u^{1/\nu})/\nu$ only holds for $u<A_{K+1}$.} 	
\begin{align*}
	\dif\Pi_{K}(u) 	&= \dif\Pi_{D,K}(u) + l \sum_{q^{\nu}<A_{K+1}}\frac{1}{\nu}\delta_{q^{\nu}}(u); \\
	\dif \pi_{K}(u) 	&= \sum_{p_{j}<A_{K+1}}\delta_{p_{j}}(u) + l\delta_{q}(u).
\end{align*}
If $x<A_{K+1}$, $N(x)=N_{K}(x)$, and applying the effective Perron formula gives that for $\kappa>1$
and $T\ge0$
\begin{align}
	\frac{1}{2}(N(x^{+})+N(x^{-})) 	&= \frac{1}{2\pi\I}\int_{\kappa-\I T}^{\kappa+\I T}\zeta_{C,K}(s)\frac{x^{s}}{s}\exp\bigl(\log\zeta_{K}(s)-\log\zeta_{C,K}(s)\bigr)\dif s \nonumber\\
							&\quad+ O\biggl(x^{\kappa}\int_{1^{-}}^{\infty}\frac{1}{u^{\kappa}\bigl(1+T\abs{\log(x/u)}\bigr)}\dif N_{K}(u)\biggr). \label{eq: effective Perron}
\end{align}
We will shift the contour of the first integral to one which is (up to some of the line segments $\Delta_{i}^{+}$) identical to the contour considered in the analysis of the continuous example $\Pi_{C}$. One can then repeat the whole analysis in Sections \ref{sec: The contribution from the saddle points} and \ref{sec: The remainder} to estimate this integral, provided that we have a good bound on $\abs{\exp(\log\zeta_{K}(s)-\log\zeta_{C,K}(s))}$, and that $\arg\bigl(\exp(\log\zeta_{K}(s)-\log\zeta_{C,K}(s))\bigr)$ is sufficiently small for $s$ on the steepest path $\Gamma_{0}$. We now show that this is the case.

Integrating by parts and using that $\dif\Pi_{K} = \dif\Pi_{C,K}$ on $[A_{K+1},\infty)$ and $\dif\Pi_{C,K}=\dif\Pi_{C}$ on $[1,A_{K+1}]$, we see that for $\sigma > 1/2$, 
\begin{align*}
	\log\zeta_{K}(s)-\log\zeta_{C,K}(s) 	&= \int_{1}^{A_{K+1}}y^{-s}\dif\,\bigl(\Pi_{K}(y)-\Pi_{C,K}(y)\bigr) \nonumber\\
								&= O(1) + \int_{1}^{A_{K+1}}y^{-s}\dif\,\bigl(\Pi_{K}(y) - \pi_{K}(y)\bigr) - \int_{1}^{A_{K+1}}y^{-s}\dif\,\bigl(\Pi_{C}(y)-\pi_{C}(y)\bigr) \\
								&\quad + \int_{1}^{A_{K+1}}y^{-\sigma}\dif\,\biggl(\sum_{p_{j}\le y}p_{j}^{-\I t} - \int_{1}^{y}u^{-\I t}\dif\pi_{C}(u)\biggr) .
\end{align*}
The bound \eqref{eq: bound exp sum} and the fact that $\dif\,(\Pi_{K} - \pi_{K})$, $\dif\,\bigl(\Pi_{C}-\pi_{C}\bigr)$ are positive measures now imply that uniformly for $\sigma \ge 3/4$\textcolor{blue}{,} say,
\begin{equation}
\label{eq: bound zeta_{K}-zeta_{C,K}}
	\abs{\log\zeta_{K}(s) - \log\zeta_{C,K}(s)} \le D \sqrt{\log(\, \abs{t}+2)},
\end{equation}
where $D>0$ is a constant which depends on the implicit constant in \eqref{eq: bound exp sum}, but which is independent of $K$. Similarly, 
\[
	(\log\zeta_{K}(s))' - (\log\zeta_{C,K}(s))' \ll \sqrt{\log(\,\abs{t}+2)}.
\]
Also, for infinitely many even and odd $K$,
\begin{align*}
	&\Im\bigl(\log\zeta_{K}(1+\I\tau_{K}) - \log\zeta_{C,K}(1+\I\tau_{K})\bigr) \\
	&= \Im\biggl\{\log\zeta_{D,K}(1+\I\tau_{K})-\log\zeta_{C,K}(1+\I\tau_{K}) - l\log(1-q^{-(1+\I\tau_{K})}) \\
	&\qquad\qquad + l\biggl(\log(1-q^{-(1+\I\tau_{K})})+\sum_{q^{\nu}<A_{K+1}}\frac{q^{-\nu(1+\I\tau_{K})}}{\nu}\biggr)\biggr\} \in \biggl[-\frac{6\pi}{160},\frac{6\pi}{160}\biggr]+2\pi\Z,
\end{align*}
by \eqref{eq: even} and \eqref{eq: odd} and since 
\[
	l\abs[4]{\log(1-q^{-(1+\I\tau_{K})})+\sum_{q^{\nu}<A_{K+1}}\frac{q^{-\nu(1+\I\tau_{K})}}{\nu}} \ll (1/q)^{\frac{\log A_{K+1}}{\log q}} < \frac{\pi}{160},
\]
say.
Let now $s\in\Gamma_{0}$, the steepest path through $s_{0}$. Then $\abs{s-(1+\I\tau_{K})} \ll \log_{2}B_{K}/\log B_{K}$, and
\begin{align*}
	\log\zeta_{K}(s) - \log\zeta_{C,K}(s) 	&= \log\zeta_{K}(1+\I\tau_{k})-\log\zeta_{C,K}(1+\I\tau_{K}) + \int_{1+\I\tau_{K}}^{s}\bigl(\log\zeta_{K}(z)-\log\zeta_{C,K}(z)\bigr)'\dif z \\
								&= \log\zeta_{K}(1+\I\tau_{k})-\log\zeta_{C,K}(1+\I\tau_{K}) + O\biggl(\sqrt{\log\tau_{K}}\frac{\log_{2}B_{K}}{\log B_{K}}\biggr),
\end{align*}
so for such $s$, 
\[
	\Im\bigl(\log\zeta_{K}(s)-\log\zeta_{C,K}(s)\bigr) \in \biggl[-\frac{7\pi}{160},\frac{7\pi}{160}\biggr] + 2\pi\Z.
\]

Since $N(x) \ll x$ (which follows for instance from Theorem \ref{th: Diamond's theorem}), there exists some $\tilde{x}_{K}\in (x_{K}-1,x_{K})$ such that 
\[
	\biggl(\tilde{x}_{K} - \frac{1}{\tilde{x}_{K}^{2}}, \tilde{x}_{K} + \frac{1}{\tilde{x}_{K}^{2}}\biggr)\cap\MN = \varnothing,
\]
where $\MN$ is the set of integers generated by $\MP$. We will apply the effective Perron formula \eqref{eq: effective Perron} with $x=\tilde{x}_{K}$ instead of $x_{K}$, in order to avoid a technical difficulty in bounding the error term in this formula. Changing $x_{K}$ to $\tilde{x}_{K}$ is not problematic, since $\sigma\log(x_{K}/\tilde{x}_{K}) \ll 1$, and on the steepest path $\Gamma_{0}$, $\Im(s\log(x_{K}/\tilde{x}_{K})) \ll \tau_{K}/x_{K} < \pi/160$ say. This implies that on the steepest path $\Gamma_{0}$ through $s_{0}$ the argument of the integrand in \eqref{eq: effective Perron} when $x=\tilde{x}_{K}$ 
belongs to
$\pi/2 + [-3\pi/10, 3\pi/10] + 2\pi\Z$ (resp.\ $\in 3\pi/2 + [-3\pi/10, 3\pi/10] + 2\pi\Z$) for infinitely many even (resp.\ odd) $K$. Together with the bound \eqref{eq: bound zeta_{K}-zeta_{C,K}} this 
yields that for infinitely many even and odd $K$ the contribution from $s_{0}$ is the same as in \eqref{eq: contribution s_{0} explicit} (but possibly with a different value for the implicit constant).
One might check that the bound \eqref{eq: bound zeta_{K}-zeta_{C,K}} is also sufficient to treat all the other pieces of the contour, except for the line segment $\Delta_{3}^{+}$. We will replace this segment together with $\Delta_{4}^{+}$ by a different contour, a little more to the left, so that $x^{s}$ can counter the additional factor $\exp(D\sqrt{\log t})$. We will also need a larger value of $T$ to bound the error term in the effective Perron formula,
 so we now take $T=(x_{K})^{4}$ instead of $T=(x_{K})^{2}$. 

Recall that $\Delta_{2}^{+}$ brought us to the point $\sigma'+\I T_{2}^{+}$. First, set $\tilde{\Delta}_{3}^{+}=[\sigma'+\I T_{2}^{+}, \sigma'+2\I\tau]$. This segment contributes $\ll x^{\sigma'}\exp(D\sqrt{\log(2\tau)})$, which is admissible. Next we want to move to the left in such a way that $\int_{s}^{\infty}\eta_{K}$ remains under control. Set $\sigma(t)=1-\log t/\log B_{K}$. If $\sigma\ge\sigma(t)$ and $t\ge2\tau_{K}$, then
\[
	\sum_{k=0}^{K}\int_{s}^{s+1}\bigl(\eta_{k}(z) + \tilde{\eta}_{k}(z) + \xi_{k}(z)\bigr)\dif z \ll \sum_{k=0}^{K}\frac{B_{k}^{1-\sigma(t)}}{t\log B_{k}} \ll \sum_{k=0}^{K}\frac{1}{\log B_{k}} \ll 1,
\]
by the rapid growth of $(B_{k})_{k}$ (see \ref{Property (a)}). Set $\tilde{\Delta}_{4}^{+}=[\sigma(2\tau)+2\I\tau, \sigma'+2\I\tau]$ (note that $\sigma(2\tau)<\sigma'$). The contribution of $\tilde{\Delta}_{4}^{+}$ is bounded by $(x^{\sigma'}/\tau)\exp(D\sqrt{\log(2\tau)})$, which is negligible. Now set $\sigma''=\sigma'-2D/\sqrt{\log x}$. 
We consider two cases.

\noindent\textbf{Case 1:} $\sigma(2\tau) \le \sigma''$, that is, $\alpha > 1/3$. Then we set $\tilde{\Delta}_{5}^{+} = [\sigma(2\tau)+2\I\tau, \sigma(2\tau)+\I x^{4}]$, its contribution is $\ll x^{\sigma''}(\log x)\exp\bigl(D\sqrt{\log x^{4}}\bigr) = x^{\sigma'}\log x$, which is admissible.

\noindent\textbf{Case 2:} $\sigma(2\tau)>\sigma''$, that is, $\alpha \leq 1/3$. Let $T_{3}^{+}$ be the solution of $\sigma(T_{3}^{+}) = \sigma''$, and set $\tilde{\Delta}_{5}^{+}=\{\sigma(t)+\I t: 2\tau\le t\le T_{3}^{+}\} \cup [\sigma''+\I T_{3}^{+}, \sigma''+\I x^{4}]$. This contributes
\[	
	\ll x\int_{2\tau}^{T_{3}^{+}}\exp\biggl(-\frac{\log x}{\log B}\log t + D\sqrt{\log t}\biggr)\frac{\dif t}{t} + x^{\sigma''}(\log x)\exp\bigl(D\sqrt{\log x^{4}}\bigr).
\] 
The first integral is bounded by
\[
	x\int_{2\tau}^{T_{3}^{+}}\exp\biggl(-\frac{\log x}{2\log B}\log t\biggr)\frac{\dif t}{t} \ll x\exp\biggl(-\frac{\log x}{2\log B}\log(2\tau)\biggr) \ll x \exp\biggl(-\frac{c\log x}{2(\log B)^{1-\alpha}}\biggr),
\]
which is again admissible.

\noindent Finally, we set $\tilde{\Delta}_{6}^{+} = [\sigma(2\tau)+\I x^{4}, 3/2 + \I x^{4}]$ or $[\sigma'' + \I x^{4}, 3/2 + \I x^{4}]$, this contributes $x^{3/2 - 4}\exp\bigl(D\sqrt{\log x^{4}}\bigr)$, which is negligible.

Next, we need to estimate the error term in the effective Perron formula
\begin{equation}
\label{eq: Perron remainder}
	x^{3/2}\int_{1^{-}}^{\infty}\frac{1}{u^{3/2}\bigl(1+x^{4}\abs{\log(x/u)}\bigr)}\dif N_{K}(u), \quad x=\tilde{x}_{K}.
\end{equation}
We have that 
\begin{multline*}
	\dif N_{K} = \expm(\dif\Pi_{K}) = \expm\biggl(\sum_{p_{j}^{\nu}<A_{K+1}}\frac{1}{\nu}\delta_{p_{j}^{\nu}} + l \sum_{q^{\nu}<A_{K+1}}\frac{1}{\nu}\delta_{q^{\nu}}\biggr) \\
						+ \expm\biggl(\sum_{p_{j}^{\nu}<A_{K+1}}\frac{1}{\nu}\delta_{p_{j}^{\nu}} + l \sum_{q^{\nu}<A_{K+1}}\frac{1}{\nu}\delta_{q^{\nu}}\biggr) \ast 
								\biggl(\chi_{[A_{K+1},\infty)}\dif\Li + \frac{1}{2}\bigl(\chi_{[A_{K+1},\infty)}\dif\Li\bigr)^{\ast 2} + \dotso\biggr) \eqqcolon \dif m_{1} + \dif m_{2}.
\end{multline*}
Since $\dif m_{1} \le \dif N$, the contribution of $\dif m_{1}$ to \eqref{eq: Perron remainder} is bounded by 
 \[
 	x^{3/2}\sum_{n\in \MN}\frac{1}{n^{3/2}\bigl(1+x^{4}\abs{\log (x/n)}\bigr)} \ll x^{3/2-4} + \sum_{\substack{n\in\MN \\ x/2\le n\le 2x}}\frac{x}{x^{4}\abs{n-x}},
\]
where we used $\abs{\log(x/n)} \gg \abs{n-x}/x$ when $x/2\le n\le 2x$. By the choice of $x=\tilde{x}_{K}$, $\abs{n-x} \ge 1/x^{2}$, so the last sum is bounded by $(1/x)N_{K}(2x)$, which is bounded. The second measure $\dif m_{2}$ has support in $[A_{K+1},\infty)$. Since we may assume that $A_{K+1}>2x_{K}$ by \ref{Property (a)} and since $\dif m_{2} \le \dif N_{K}$, the contribution of $\dif m_{2}$ to \eqref{eq: Perron remainder} is bounded by
\[
	\frac{1}{x^{4}}\int_{A_{K+1}}^{\infty}\frac{\dif N_{K}(u)}{u^{3/2}} \ll \frac{1}{x^{4}}.
\]
(The integral is bounded by $\zeta_{K}(3/2)$, which is bounded independent of $K$.)

To complete the proof, it remains to bound $\rho-\rho_{K}$, where $\rho$ and  $\rho_{K}$ are the asymptotic densities of $N$ and $N_{K}$, respectively.
We have
\begin{align*}
\log \rho - \log \rho_{K} 	&= \int_{1^{-}}^{\infty}\frac{1}{u}\biggl(\sum_{p_{j}^{\nu}\ge A_{K+1}}\frac{1}{\nu}\delta_{p_{j}^{\nu}}(u) 
							+ l\sum_{q^{\nu}\ge A_{K+1}}\frac{1}{\nu}\delta_{q^{\nu}}(u) - \chi_{[A_{K+1},\infty)}\dif\Li(u)\biggr)\\
					&\ll \int_{A_{K+1}}^{\infty}\frac{1}{u^{2}}\abs{\Pi(u)-\Pi(A_{K+1}^{-})-\Li(u)+\Li(A_{K+1})}\dif u\\
					&\ll \int_{A_{K+1}}^{\infty}\frac{\exp\bigl(-c(\log u)^{\alpha}\bigr)}{u}\dif u \ll \exp\bigl(-(c/2)(\log A_{K+1})^{\alpha}\bigr)\le \frac{1}{x_{K}},		
\end{align*}
where we may assume the last bound in view of \ref{Property (a)}. In conclusion, we have that (on some subsequence containing infinitely many even and odd $K$):
\begin{align*}
N(\tilde{x}_{K}) - \rho\tilde{x}_{K} 	&= N_{K}(\tilde{x}_{K})-\rho_{K}\tilde{x}_{K} + (\rho-\rho_{K})\tilde{x}_{K} \\
							&= \Omega_{\pm}\Bigl(\tilde{x}_{K}\exp\bigl(-(c(\alpha+1))^{\frac{1}{\alpha+1}}(\log \tilde{x}_{K}\log_{2}\tilde{x}_{K})^{\frac{\alpha}{\alpha+1}}(1+\dotso)\bigr)\Bigr) + O(1).
\end{align*}

\end{document}